\pdfoutput=1 

\documentclass{article}

\usepackage{amscd,amssymb,verbatim,graphicx,stmaryrd,amsthm,amsmath,color,mathtools}

\usepackage[colorlinks]{hyperref}
\usepackage {tikz}
\usetikzlibrary {positioning}

\definecolor {processblue}{cmyk}{0.96,0,0,0}

\DeclarePairedDelimiterX{\inp}[2]{\langle}{\rangle}{#1, #2}

\usepackage[ps,all,color]{xy}

\input xy    
\xyoption{curve}  
\xyoption{arc} 
\xyoption{frame} 
\xyoption{tips} 
\xyoption{arrow} 
\xyoption{color} 
\xyoption{line}
\xyoption{poly}

\UseCrayolaColors

\usepackage{mathpazo} 

\newtheorem{theorem}{Theorem}[section]  
\newtheorem{lemma}[theorem]{Lemma}  
\newtheorem{proposition}[theorem]{Proposition}
\newtheorem{corollary}[theorem]{Corollary}

\theoremstyle{definition}

  \newtheorem{example}[theorem]{Example}
  \newtheorem{remark}[theorem]{Remark}
\newtheorem{notation}[theorem]{Notation}

\newcommand{\sumdd}[2]{\displaystyle \sum_{#1}^{#2}}
\newcommand{\sumd}[1]{\displaystyle \sum_{#1}}

\newcommand{\fracd}[2]{\displaystyle \frac{#1}{#2}}
\newcommand{\bb}[1]{\mathbb{#1}}

\newcommand{\defeq}{\mathrel{\mathpalette{\vcenter{\hbox{$:$}}}=}}

\newcommand{\F}{{\mathcal{F}}}

\newcommand{\T}{{\mathcal{T}}}

\newcommand{\V}{{\mathcal{V}}}
\newcommand{\E}{{\mathcal{E}}}
\newcommand{\Ss}{{\mathcal{S}}}

\title{Characterizing immutable sandpiles: \\ A first look}

\author{David L. Duncan and Wesley J. Engelbrecht}

\date{}

\begin{document}

\maketitle

\begin{abstract}
By working with coefficients in $\bb{Z}$ or $\bb{R}$, one can define two different notions of stability for a sandpile on a graph. We call a sandpile \emph{immutable} when these notions agree. Our main results give linear-algebraic characterizations for large classes of immutable sandpiles.
\end{abstract}

\tableofcontents

\section{Introduction}

Let $\Gamma$ be a connected, finite multigraph without self-loops. Fix a vertex $v_*$, called the \emph{sink}, and denote by $\V'$ the set of non-sink vertices of $\Gamma$. An \emph{(abelian) sandpile} is a non-negative element of the group $\bb{Z}^{\V'}$ of integer-labelings of $\V'$. Intuitively, one can view a sandpile as specifying a configuration of sand particles placed on the non-sink vertices. Sandpiles were originally introduced by Bak--Tang--Wiesenfeld \cite{Bak1,Bak2}, who showed that sandpiles exhibit a structure rich enough to model self-organizing criticality on lattices. This was later generalized to arbitrary graphs by Dhar \cite{Dhar}, with further combinatorial ties elucidated by numerous researchers, including Bj\"{o}rner--Lov\'{a}sz--Shor \cite{BLS} and Biggs \cite{NLB}, who introduced the sink and much of the perspective we take here.

The aforementioned rich structure of sandpiles arises through an equivalence relation called ``legal toppling''. To describe this, let $d: \V' \rightarrow \bb{Z}$ be the degree (valance) map. Given $\sigma \in \bb{Z}^{\V'}$, a vertex $v \in \V'$ is \emph{unstable (for $\sigma$)} if $\sigma(v) \geq d(v)$. Then $\sigma$ is \emph{unstable} if it has an unstable vertex; otherwise $\sigma$ is called \emph{stable}. A \emph{toppling} of a vertex $v$ is the new labeling $\sigma' \in \bb{Z}^{\V'}$ obtained from $\sigma$ by reducing by $d(v)$ the value of $\sigma(v)$, and by increasing by 1 the value of $\sigma$ at each vertex adjacent to $v$. In the sand analogy, this is as if one vertex spills a grain of sand to each of its neighbors. If $v$ is unstable for $\sigma$, then the toppling $\sigma'$ is called \emph{legal}. There is a unique stable sandpile $\tau_\sigma$ that is obtained from $\sigma$ by a sequence of legal topplings \cite{BLS}. We call $\tau_\sigma$ the \emph{stabilization} of $\sigma$. The \emph{$\bb{Z}$-odometer of $\sigma$} is the sandpile $u_{\sigma}^{\bb{Z}}$ defined so that $u_{\sigma}^{\bb{Z}}(v)$ is the number of times the vertex $v \in \V'$ is toppled during the stabilization of $\sigma$; i.e., it measures the amount of sand released by each vertex in this stabilization process.

It was observed by Fey--Levine--Peres \cite{FLP} that this has an equivalent least action-type characterization in which the $\bb{Z}$-odometer is realized as the unique minimizer of a certain system of matrix inequalities; we review this in Section \ref{sec:TheOdometer}. In this characterization, the minimization is taking place over the space $\bb{Z}^{\V'}$ of $\bb{Z}$-labelings. In an effort to better understand scaling limits of sandpiles on integer lattices, Levine and Peres \cite{LP2} considered the same system of matrix inequalities, but minimized over the larger space $\bb{R}^{\V'}$ of $\bb{R}$-labelings. In this larger system, there too is a unique minimizer $u_{\sigma}^{\bb{R}}: \V' \rightarrow \bb{R}$, which we call the \emph{$\bb{R}$-odometer of $\sigma$}. The idea here is that if $\sigma$ were allowed to topple through non-integral quantities, then it may stabilize to a configuration different than $\tau_\sigma$. This can happen, and it turns out that the two stabilizations (integral and continuous) often exhibit very different qualitative features \cite{LP2,LP}. 

There is a physical analogy for the continuous setting too: In place of the grains of sand, imagine a tower of a viscous liquid on each vertex, with mass-distribution described by $\sigma$. Suppose there were a vertex with too much fluid (more than the degree minus one). Then, over time, the fluid at this vertex would flow uniformly to its neighbors, until the vertex becomes stable in this new continuous sense. The $\bb{R}$-odometer measures the amount of fluid released by each vertex in this ``$\bb{R}$-stabilization'' process. 

We will say a sandpile $\sigma$ is \emph{immutable} if $u_{\sigma}^{\bb{Z}} = u_{\sigma}^{\bb{R}}$; that is, if the two notions of stabilization for $\sigma$ agree. Otherwise, we will call the sandpile \emph{mutable}. The idea is that a mutable configuration is one for which a change of state (e.g., from solid to liquid) can change its stabilization. To visualize this, the above analogies can be refined as follows: In place of sand particles, suppose now one has identical, homogenous metal spheres. Start with a sandpile consisting of such metal spheres, as well as one duplicate copy of that sandpile. Allow the first sandpile to stabilize in the usual sandpile sense. For the other sandpile, first increase the temperature to slightly beyond the melting point of the metal, and then allow it to stabilize. The initial sandpile is immutable if and only if the resulting stable configurations are identical.

The aim of the present paper is to better understand which sandpiles are immutable, and which are not. Our efforts were motivated by the following questions:

\begin{itemize}
\item[Q1.] Do immutable sandpiles always exist? How about mutable sandpiles?
\item[Q2.] Is it more common for a sandpile to be mutable or immutable?
\item[Q3.] Is there a simple criterion for determining when a sandpile is immutable?
\item[Q4.] What does the set of all immutable sandpiles look like? Does it have any interesting structure?
\end{itemize}

We will see in Example \ref{ex:stable} that every stable sandpile is immutable, so this provides an affirmative answer to the first question in Q1 (e.g., $\sigma = 0$ and $\sigma = d-1$ are both immutable). Our starting point for the remaining questions in Q1--Q4 is the following theorem, which provides sufficient conditions relative to which immutability is equivalent to the integrality of a certain vector quantity. To state it, let $L'$ be the reduced Laplacian (see Section \ref{sec:Laplacian}) and define a sandpile $\sigma $ to be \emph{uniformly large} if $\sigma \geq d - 1$. 

\begin{theorem}\label{thm:1}
Suppose $\sigma$ is uniformly large. Then the $\bb{R}$-odometer is $u_\sigma^{\bb{R}} = (L')^{-1}(\sigma - d + 1)$. Moreover, $\sigma$ is immutable if and only if 
$$(L')^{-1}(\sigma - d + 1) \in \bb{Z}^{\V'}$$ 
is integral. 
\end{theorem}

We prove this theorem in Section \ref{sec:ProofOfThm1.1}. There, we also give examples in which the conclusions of the theorem do not hold in the absence of the hypothesis that $\sigma$ is uniformly large. 

With the characterization given in Theorem \ref{thm:1}, we now begin to address Q1--Q4 above. For example, the existence of mutable sandpiles in Q1 is deduced for a very large collection of multigraphs by the following corollary, whose proof is given in Section \ref{sec:ProofOfCor1.1}.

\begin{corollary}\label{cor:1}
Assume there is a vertex $v$ in $\Gamma$ that is adjacent to $v_*$ and satisfies the following:
    \begin{enumerate}
      \item[(a)] if $w \neq v$, then $w$ can be connected to $v_*$ by an edge-path that does not contain $v$;
    \item[(b)] $v$ has degree $d(v) \geq 2$.
    \end{enumerate}
 Then there is a sandpile on $\Gamma$ that is mutable. 
\end{corollary}

Our proof uses the hypotheses on $v$ to construct a uniformly large sandpile $\sigma$ for which $(L')^{-1}(\sigma - d+1)$ is not integral, at which point the result is immediate from Theorem \ref{thm:1}. That said, these hypotheses are by no means necessary (see Example \ref{ex:path4}) and it seems it is rare for graphs to only admit immutable sandpiles. Nevertheless, these hypotheses cannot be altogether dropped either. For example, we will see in Examples \ref{ex:multik2} and \ref{ex:immutablepath} that when $\Gamma = P_k$ is the path on $k = 2$ or $k = 3$ vertices, then there are no mutable sandpiles. These graphs fail the hypotheses of the corollary since every vertex on $P_2$ fails (b), while every vertex on $P_3$ fails either (a) or (b).

The questions in Q2--Q4 are more global in nature, and so are addressed by Theorem \ref{thm:1} only for the class of sandpiles that are uniformly large. Nevertheless, within this class of sandpiles, Theorem \ref{thm:1} is entirely satisfactory. Indeed, for Q3, the integrality of $(L')^{-1}(\sigma - d + 1)$ is a theoretically-pleasing condition for immutability, certainly when compared to the definition itself (in which one minimizes over solutions of a system of matrix inequalities). For Q4, the set of immutable, uniformly large sandpiles can be identified with the set of $w \in \bb{Z}^{\V'}$ with $L'(w) \geq 0$; that is, it can be identified with the integer points of a polytope. We turn finally to Q2. This question is a heuristic one, so we give it a heuristic answer: Except for rare cases, the inverse reduced Laplacian $(L')^{-1}$ is non-integral. Moreover, for a sandpile $\sigma$ on any such $\Gamma$, the integrality of the quantity $(L')^{-1}(\sigma - d + 1)$ is generally much less likely than its non-integrality (by roughly a factor of the largest reduced denominator appearing in $(L')^{-1}$). That is, with the exception of a few special cases, it is much more likely for a uniformly large sandpile to be mutable than for it to be immutable.

Our next result restricts to a special class of graphs whose additional structure provides more refined information than that which is afforded by Theorem \ref{thm:1}; see Section \ref{sec:Proof3} for a proof.

\begin{corollary}\label{cor:2} 
Suppose $\Gamma$ is the cone of a regular graph, and choose the sink to be the cone point. Let $\sigma \in \bb{Z}^{\V'}$.
\begin{enumerate}
\item[(a)] Assume $\sigma$ is immutable and uniformly large. Then $\sigma = L' a$ for some uniformly large $a \in \bb{Z}^{\V'}$. 
\item[(b)] Assume $\sigma \geq 0$ and $\sigma = L' a$ for some uniformly large $a \in \bb{Z}^{\V'}$. Then $\sigma$ is immutable.
\end{enumerate}
In both cases, $u_\sigma^{\bb{R}} = a - d +1$. 
\end{corollary}

With $\Gamma$ as in Corollary \ref{cor:2}, this corollary makes an unexpected tie with the critical group $K(\Gamma) \defeq \mathrm{coker}(L': \bb{Z}^{\V'} \rightarrow \bb{Z}^{\V'})$: If $\sigma$ is uniformly large, then $\sigma$ is immutable if and only if $\sigma$ is a representative of the identity in $K(\Gamma)$.

The characterization of immutable sandpiles in Theorem \ref{thm:1} and Corollary \ref{cor:2} essentially requires inverting the reduced Laplacian $L'$. Of course, the theoretical and computational tools available for this are vast; however, these strategies can be computationally taxing since inverting matrices is a highly non-linear operation. It would be practically convenient to have a criterion for immutability that one could check directly, without having to invert a matrix. This would also be conceptually pleasing since, despite the elegance of the immutability criterion expressed in these results, precisely which class of sandpiles this identifies remains partially hidden behind the veil of inverting $L'$. Towards this end, in Section \ref{sec:Examples}, we give direct criteria for immutability of uniformly large sandpiles on trees, complete graphs, and wheel graphs. Our computations in Section \ref{sec:Examples} are rooted in an extension of Kirchoff's matrix tree theorem (see Section \ref{sec:Laplacian}) that gives a spanning tree-type interpretation of an arbitrary minor of the Laplacian matrix.

\medskip

Through the results and examples mentioned above, we gain a fairly detailed picture of immutable and mutable sandpiles that are either uniformly large or stable (which can be viewed as ``uniformly small''). Ideally, one would have a characterization of immutability in the absence of our uniformly large hypothesis (ULH) of Theorem \ref{thm:1}. Corollary \ref{cor:2} (b) (cf. Example \ref{ex:expanded}), and Examples \ref{ex:multik2} and \ref{ex:immutablepath} give results in this direction, but their scope is limited. As we have suggested above, throughout this paper, we probe the question of the necessity of our hypotheses through various examples that were particularly chosen to be simple, yet instructive. The moral of these examples is that the characterizations of immutability given in Theorem \ref{thm:1} and Corollary \ref{cor:2} would need to be significantly altered if the ULH were dropped. As such, the class of sandpiles that are neither uniformly large nor stable is not deeply explored here, and we leave for future work a more complete investigation of immutability for this class. We note also that it is precisely this class of sandpiles that produce the beautiful pictures in \cite{LP2,LP}; indeed, these pictures are the stabilizations of mutable sandpiles that are point-masses on compactifications of $\bb{Z}^n$.

\medskip

\noindent {\bf Acknowledgements.} The authors would like to thank Joshua Ducey for introducing us to sandpiles and for his help with early drafts of this manuscript. Thanks are also due to the 2019 REU Team at James Madison University: Jawahar Madan, Eric Piato, Christina Shatford, and Angela Vichitbandha.

This work was partially supported by the Jeffrey E. Tickle '90 Family Endowment in Science \& Mathematics, and by NSF Grant Number NSF-DMS 1560151.

\section{The Laplacian and its minors}\label{sec:Laplacian}

Let $\Gamma$ and $v_*$ be as in the introduction, and write $\V = \V_\Gamma$ for the vertex set of $\Gamma$. For $u \in \bb{Z}^\V$, define $Lu: \V \rightarrow \bb{Z}$ to be the function
$$(L u)(v) = \sumd{w \in \V} L_{vw} u(w), \indent L_{vw} = \left\{ \begin{array}{ll}
																															d(v) & \textrm{if $v = w$}\\
																															-\#(v, w) & \textrm{if $v \neq w$}
																															\end{array}\right.$$
where $\#(v, w) \in \bb{Z}$ is the number of edges connecting $v$ and $w$. The association $u \mapsto Lu$ defines a group homomorphism $L: \bb{Z}^{\V} \rightarrow \bb{Z}^{\V}$, which is the \emph{Laplacian} of $\Gamma$. Let $\pi: \bb{Z}^\V \rightarrow \bb{Z}^{\V'}$ be the projection, where $\V' = \V \backslash \left\{v_* \right\}$. Since $\Gamma$ is connected, this projection restricts to the subset 
$$\bb{Z}^{\V}_0 \defeq \Big\{u \in \bb{Z}^\V \; \Big| \; \sumd{v \in \V} u(v) = 0 \Big\}$$
to yield a group isomorphism $\pi \vert: \bb{Z}^{\V}_0 {\rightarrow} \bb{Z}^{\V'}$. Denote by $\iota: \bb{Z}^{\V'} \rightarrow \bb{Z}^{\V}_0 \subseteq \bb{Z}^\V$ the inverse of this isomorphism. The \emph{reduced Laplacian} is the composition
$$L' = L'_\Gamma \defeq \pi \circ L \circ \iota: \bb{Z}^{\V'} \longrightarrow \bb{Z}^{\V'}.$$

Suppose $G$ is an abelian group. Viewing $G$ as a $\bb{Z}$-module, we can interpret $L$ as defining a group homomorphism $G^{\V} \rightarrow G^{\V}$ that we denote by the same symbol $L$; similar statements hold for $L'$. When we wish to emphasize the group $G$, we will say we are \emph{working over $G$}. When working over $\bb{R}$, the operator $L'$ is invertible (e.g., see Corollary \ref{cor:kirchoff}). Thus, when working over any subgroup $G \subseteq \bb{R}$, the operator $L'$ is injective. 

We assume $\Gamma$ has $n+1$ vertices with $n \geq 1$. Given an enumeration $v_0 = v_*, v_1 , \ldots, v_n$ of the vertices, we obtain identifications $G^{\V} \cong G^{n+1}$ and $G^{\V'} \cong G^n$, with the projection $G^{\V} \rightarrow G^{\V'}$ corresponding to the projection $G^{n+1} \rightarrow G^n$ to the last $n$ components. Relative to these identifications, we can view $L$ (resp. $L'$) as an $(n+1) \times (n+1) $-matrix (resp. $n \times n$-matrix) with integer entries. The matrix $L'$ is obtained from $L$ be deleting the first row and the first column. 
 
 For the rest of this section, we fix an enumeration of the vertices of $\Gamma$ and we work over $\bb{R}$. Our main goal is to understand the components of the inverse matrix $(L')^{-1}$; this is formalized in Corollary \ref{cor:L'inverse}. Our approach is to prove an extension of Kirchoff's matrix tree theorem, Theorem \ref{thm:kirchoff}, from which Corollary \ref{cor:L'inverse} will readily follow.

\subsection{A generalization of Kirchoff's matrix tree theorem}\label{sec:Kirchoff}

A \emph{$k$-forest} is a collection $\left\{\T_0, \ldots, \T_{k-1} \right\}$ of pairwise disjoint subgraphs $\T_i$ of $\Gamma$ so that each $\T_i$ is a (connected) tree with at least one vertex. We will say a $k$-forest is \emph{spanning} if each vertex of $\Gamma$ lies in one of the trees. (We caution the reader that our use of the term ``spanning $k$-forest'' is a weaker notion than what is often meant by the term ``spanning forest''.) A spanning 1-forest is the same as a spanning tree, and any spanning $k$-forest determines a spanning $(k+1)$-forest by deleting an edge (assuming the $k$-forest has at least one edge). Since $\Gamma$ is connected with $n+1$ vertices, it follows that a $k$-forest is spanning if and only if there are $n+1-k$ edges appearing in the union of the trees. A spanning $(n+1)$-forest consists only of the vertices, each viewed as a tree with no edges. There are no spanning $k$-forests for $k > n+1$.  

The following notation will help expedite our discussion below of spanning 1- and 2-forests, which are the most important cases for us.

\begin{notation}\label{def:spanningsets}
Let $\Ss_1$ be the set of all spanning 1-forests (spanning trees) of $\Gamma$.

\medskip

Fix $v_* \in \V$ and $v, w \in \V'$. Let $\Ss_2(v, w)$ be the set of ordered pairs $(\T_0, \T_1)$, where $\{\T_0, \T_1 \}$ is a spanning 2-forest of $\Gamma$, $v_* \in \T_0$, and $v, w \in \T_1$. \hfill $\Diamond$
\end{notation}

Let $M$ be a $k \times \ell$ matrix and $0 \leq r \leq \min(k, \ell)$. Fix subsets 
$$E \subseteq \left[ k \right] , \indent V \subseteq \left[\ell \right]$$ 
each with a complement of size $ r$. Denote by 
$$M_E^V$$ 
the $r \times r$ matrix obtained from $M$ by deleting all rows indexed by $E$ and all columns indexed by $V$. (Our slightly odd notation convention is adopted to mitigate the need for even more cumbersome notation below.) If either $E$ or $V$ is empty, then we drop it from the notation. 

\begin{example}\label{ex:cof}
The $(i, j)$-minor of a matrix $M$ is often written as $M_{i,j}$. In terms of the above row/column deletion notation, this is
$$M_{i,j} = \det\left(M^{\left\{j \right\}}_{\left\{i \right\}}\right).$$
\hfill $\Diamond$ \end{example}

We will first apply this notation to the case where $M = L$ and $M = L'$, and it is worth discussing our indexing conventions in these cases. Towards this end, we note that $L$ has rows and columns indexed by $\left\{v_0, \ldots, v_n \right\} \cong \left[n+1 \right] = \left\{1, \ldots, n +1 \right\}$, so $v_i$ corresponds to the $(i+1)$st row or column. Similarly, the matrix $L'$ has rows and columns indexed by $\left\{v_1, \ldots,  v_n \right\} \cong \left[ n \right] = \left\{1, \ldots, n \right\}$, so $v_i$ corresponds to the $i$th row or column. 

\begin{example}\label{ex:L'}
The reduced Laplacian
$$L' = L^{\left\{1 \right\}}_{\left\{1 \right\}}$$
is obtained from $L$ by deleting first row and first column. The $(i, j)$-minor $L'_{i,j}$ of $L'$ can be written in terms of $L'$ or $L$ as
$$L'_{i, j} = \det\left((L')^{\left\{j \right\}}_{\left\{i \right\}}\right) = \det\left(L^{\left\{1, j+1 \right\}}_{\left\{1, i+1 \right\}}\right)$$
\hfill $\Diamond$ \end{example}

Now we can state our main result of this section.

\begin{theorem}\label{thm:kirchoff}
Fix $1 \leq r \leq n = \vert \V \vert - 1$, as well as subsets $V, W \subseteq \V$ having complements of size $r$. Then $\vert \det(L^V_W) \vert$ is the number of spanning $\vert V \vert$-forests $\big\{ \T_0, \ldots, \T_{\vert V \vert - 1} \big\}$, where each $\T_i$ contains exactly one element of $V$ and exactly one element of $W$.
\end{theorem}

\begin{remark}\label{rem:sign}
The sign of $\det(L^V_W)$ is given in (\ref{eq:sign}). For reference below, we highlight two special cases. The first is the case in which $V= W$ are equal. Then 
\begin{equation}\label{eq:VV}
\det(L^V_V) = \vert \det(L^V_V)\vert
\end{equation}
is non-negative. The other special case is if $V = \left\{v_0,v_j\right\} \cong \left\{1, j+1 \right\}$ and $W = \left\{v_0,v_i\right\} \cong \left\{1, i+1 \right\}$ have size 2, and share at least one vertex in common. Then the sign of $\det(L^V_W)$ is given by 
\begin{equation}\label{eq:ij}
\det(L^V_W) = (-1)^{i+j}\vert \det(L^V_W)\vert.
\end{equation}
\hfill $\Diamond$
\end{remark}

We prove Theorem \ref{thm:kirchoff} and Remark \ref{rem:sign} in Section \ref{sec:ProofOfThmKirch}, after we give two corollaries. The first corollary of Theorem \ref{thm:kirchoff} is the following classic result.

\begin{corollary}\label{cor:kirchoff}[Kirchoff's Matrix Tree Theorem]
The determinant $\det(L') = \vert \Ss_1 \vert$ is equal to the number of spanning trees of $\Gamma$. In particular, $\det(L') \geq 1$.
\end{corollary}

\begin{proof}[Proof of Corollary \ref{cor:kirchoff}]
The corollary is immediate from Theorem \ref{thm:kirchoff} and Remark \ref{rem:sign} with $V =W =  \left\{v_0 \right\}$ consisting only of the sink. 
\end{proof}

The next corollary gives the above-mentioned graph-theoretic interpretation of the entries of the inverse of $L'$. 

\begin{corollary}\label{cor:L'inverse}
The $(i,j)$-component of $(L')^{-1}$ is $(L')^{-1}_{ij} = \det(L')^{-1} \vert \Ss_2(v_i, v_j)\vert$. 
\end{corollary}

\begin{remark}\label{Rem:L'inverse}
Corollary \ref{cor:L'inverse} implies that the entries of $(L')^{-1}$ are non-negative, which is an observation that has been known for some time \cite{Fiedler} \cite[Ch. 6]{BerPlem} and will turn out to be crucial in our analysis below. 
\hfill $\Diamond$ \end{remark}

\begin{proof}[Proof of Corollary \ref{cor:L'inverse}]
By the cofactor expansion formula, the $(i, j)$-component of the inverse of $L'$ is given by $(-1)^{i+j} (\det(L'))^{-1} L'_{i,j}$, where $L'_{i, j}$ is the minor. The result now follows from Example \ref{ex:L'}, Remark \ref{rem:sign}, Theorem \ref{thm:kirchoff}, and the definition of $\Ss_2(v_i, v_j)$. 
\end{proof}

\begin{example}\label{path irl} 
Consider the case where $\Gamma = P_{n+1}$ is a path on $n+1$ vertices, with sink $v_*$ at one of the two endpoints. Order the remaining vertices $v_1, \ldots, v_n$ linearly, with $v_1$ adjacent to $v_*$. We claim that
\begin{equation} \label{irl path}
(L')^{-1} =  \begin{pmatrix}
						1 & 1 & 1 & \dots & 1 \\
						1 & 2 & 2 & \dots & 2 \\
						1 & 2 & 3 & \dots & 3 \\
						\vdots & \vdots & \vdots & \ddots & \vdots \\
						1 & 2 & 3 & \dots & n
						\end{pmatrix}
\end{equation}
To see this, first note that by Kirchoff's Matrix Tree Theorem, we have $\det(L') = 1$. It follows from Corollary \ref{cor:L'inverse} that the $(i, j)$-component of the inverse of $L'$ equals $\vert \Ss_2(v_i, v_j) \vert$. We will show that $\vert \Ss_2(v_i, v_j) \vert = \min(i, j)$. Indeed, fixing $( \T_0, \T_1 ) \in \Ss_2(v_i, v_j)$, we observe the following:
\begin{enumerate}
    \item The tuple $(\T_0, \T_1 )$ is uniquely determined by the vertex in $\T_1$ with the smallest index.
    \item If $\min(i,j) \leq \ell \leq n$, then the vertex $v_\ell$ is in $\T_{1}$. In particular, the smallest-index vertex in $\T_1$ has index at most $\min(i, j)$.
\end{enumerate}
It follows that the assignment
    \[ f((\T_{0},\T_{1})) = \min(\ell\; \vert \; v_{\ell} \in \T_{1}) \]
    gives a well-defined map $f: \Ss_2(v_i, v_j) \rightarrow \left\{1, \ldots, \min(i, j) \right\}$. See Figure \ref{figure1}. It is not hard to see that $f$ is a bijection, so $(L')^{-1}_{ij} = \min(i, j)$ as desired. 
\hfill $\Diamond$ \end{example}

\bigskip

\begin{figure}[h]
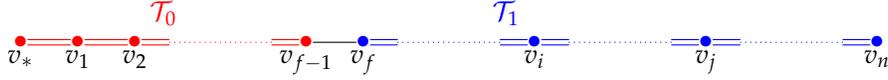

\centerline{\xy /r1.8pc/:,
(0,0)*={\xy
	(-6,0)*[red]{\bullet}="A0";
		(-5,0)*[red]{\bullet}="A1";
		(-4,0)*[red]{\bullet}="A2";
		(-3.4,0)*[red]{}="A3";
		(-1.6,0)*[red]{}="A4";
		(-1,0)*[red]{\bullet}="A5";
		(0,0)*[blue]{\bullet}="A6";
		(.6,0)*[blue]{}="A7";
		(2.4,0)*[blue]{}="A8-";
		(3,0)*[blue]{\bullet}="A8";
		(3.6,0)*[blue]{}="A8+";
		(5.4,0)*[blue]{}="A9-";
		(6,0)*[blue]{\bullet}="A9";
		(6.6,0)*[blue]{}="A9+";
		(8.4,0)*[blue]{}="A10";
		(9,0)*[blue]{\bullet}="A11";
		"A0"*+!U{\hbox{${{v_*}}$}};
		"A1"*+!U{\hbox{${{v_1}}$}};
		"A2"*+!U{\hbox{${{v_2}}$}};
		"A5"*+!U{\hbox{${{v_{f-1}}}$}};
		"A6"*+!U{\hbox{${{v_{f}}}$}};
		"A8"*+!U{\hbox{${{v_{i}}}$}};
		"A9"*+!U{\hbox{${{v_{j}}}$}};
		"A11"*+!U{\hbox{${{v_{n}}}$}};
		"A0";"A1" **[red]\dir2{-};
		"A1";"A2" **[red]\dir2{-};
		"A2";"A3" **[red]\dir2{-};
		"A3";"A4" **[red]\dir{.};
		"A4";"A5" **[red]\dir2{-};
		"A5";"A6" **\dir{-};
		"A6";"A7" **[blue]\dir2{-};
		"A7";"A8-" **[blue]\dir{.};
		"A8-";"A8" **[blue]\dir2{-};
		"A8";"A8+" **[blue]\dir2{-};
		"A8+";"A9-" **[blue]\dir{.};
		"A9-";"A9" **[blue]\dir2{-};
		"A9";"A9+" **[blue]\dir2{-};
		"A9+";"A10" **[blue]\dir{.};
		"A10";"A11" **[blue]\dir2{-};
(-3.5,.5)*[red]{\T_0};
(2.5,.5)*[blue]{\T_1};
\endxy};
\endxy}
\vspace{.2cm}
\caption{Pictured here is the case $f < i \leq j$, where we have set $f \defeq f((\T_{0},\T_{1}))$. The edges of each tree have been doubled for emphasis. The vertices and edges of the tree $\T_0$ are indicated in red, and those of the tree $\T_1$ are in blue; note that the edge from $v_{f-1}$ to $v_f$ is not contained in either tree. \hfill $\Diamond$}
\label{figure1}
\end{figure}

\subsection{Proof of Theorem \ref{thm:kirchoff}}\label{sec:ProofOfThmKirch}

Fix an enumeration on the edge set $\E$, and an orientation on each edge in $\E$. Let $B: \bb{Z}^\V \rightarrow \bb{Z}^\E$ be the associated incidence matrix (note that our convention is that the columns of $B$ are indexed by the vertices $\V$ and rows are indexed by the edges $\E$). Our main tool for proving Theorem \ref{thm:kirchoff} is the following.

\begin{lemma}\label{lem:kirchoff}
Fix $1 \leq r \leq \min(\vert \E \vert, \vert \V \vert)$, as well as subsets $E \subseteq \E$ and $V \subseteq \V$ with complements of size $r$. Then the following are equivalent for the incidence matrix $B$:
\begin{itemize}
\item[$(a)$] $\det(B^V_E) = \pm 1$;
\item[$(b)$] $\det(B^V_E) \neq 0$;
\item[$(c)$] The edges of $\E \backslash E$ are those of a spanning $\vert V \vert$-forest $\big\{\T_0, \ldots, \T_{\vert V \vert -1} \big\}$, where each $\T_i$ contains exactly one vertex of $V$. 
\end{itemize}
\end{lemma}

\begin{proof}
Clearly $(a) \Rightarrow (b)$. Next we will show that $(c) \Rightarrow (a)$. For this, let $\big\{\T_0, \ldots, \T_{\vert V \vert-1} \big\}$ be as in (c). We may reorder the vertices and edges so that first those in the tree $\T_0$ appear in the ordering, then those in $\T_1$ appear, and so forth. Let $B_i$ be the incidence matrix of $\T_i$, and let $D_i$ be the (square) matrix obtained from $B_i$ by deleting the column corresponding to the unique element of $V$ that $\T_i$ contains. Since the $\T_i$ are disjoint, it follows that $B_E^V$ is a block diagonal matrix of the form
					     \begin{equation}\label{eq:mstform}
					     B_E^V = \left(\begin{array}{cccc}
					     				D_0 &   &  & \\
					     				     & D_1 &  &\\\
					     				      &       & \ddots &   \\
					     				                &     &    & D_{\vert V \vert -1}
					     				                \end{array}\right).
					     				     \end{equation}
					     				   A simple induction on the size of each tree $\T_i$ shows that $\det(D_i) = \pm 1$. This proves $(a)$.

\begin{remark}\label{rem:1}
It is possible to determine the sign of $\det(D_i)$ from the following general observation: Let $B(T)$ be the incidence matrix for a tree $T$ with vertices labeled by $\left\{w_1, \ldots, w_\ell \right\}$. Given a vertex $w$, let $D_w$ be the matrix obtained from $B(T)$ by deleting the column corresponding to $w$. Then $\det(D_{w_i}) = (-1)^{i+j} \det(D_{w_j})$. 
\hfill $\Diamond$ 
\end{remark}

Now we will show $(b) \Rightarrow (c)$. For this, assume $\det(B_E^V) \neq 0$. We will first show that the edges of $\E \backslash E$ form a disjoint union of trees. If not, then there are edges in $\E \backslash E$ that form a cycle. After possibly reordering the vertices and edges and multiplying some rows by $-1$, we may assume the matrix $B_E^V$ is of the form
$$B_E^V = \left( \begin{array}{cc}
					C & 0\\
					* & *
					\end{array} \right)$$
					   where $C$ is the incidence matrix for the cycle. (The zero in the upper right reflects the fact that each edge in the cycle has both of its vertices within the cycle.) It is easy to check that the sum of the column vectors in $C$ is zero, so the matrix $C$ has determinant zero. Since $C$ is a square matrix, this implies that $B_E^V$ has determinant zero, which is a contradiction. 
					     
					     We may therefore assume that the edges of $\E \backslash E$ form a disjoint union of trees. There may be vertices of $\Gamma$ that do not appear in any of these trees. However, by including each such vertex as an additional tree (with one vertex and no edges), we obtain a spanning $k$-forest $\big\{\T_0, \ldots, \T_{k-1} \big\}$, for some $k \geq 1$, with the same underlying edge set as $\E \backslash E$. Moreover, we have
					     $$n+1 - k = \vert \E \backslash E \vert = \vert \V \backslash V \vert = n+1 - \vert V \vert.$$
					     It follows that $k = \vert \V \vert$. We will show that each $\T_i$ contains exactly one element of the vertex deletion set $V$; since $k = \vert \V \vert$, it will follow immediately from this that each $\T_i$ contains exactly one element of $V$, and hence (c) follows. 
					     
					     Suppose that there is a tree $\T_i$ that does not contain any vertex from $V$. Relabel the vertices and edges so those of $\T_i$ appear first in this new labeling, and let $k \geq 1$ denote the number of vertices in $\T_i$. First assume that $k > 1$ and let $B(\T_i)$ be the $(k-1) \times k$ incidence matrix of $\T_i$. Since $\T_i$ is disjoint from the other trees, it follows that the incidence matrix has the form
						\begin{equation}\label{eq:MST}
						B_E^V = \left( \begin{array}{cc}
					B(\T_i) & 0\\
					0 & *
					\end{array} \right).	
					\end{equation}
					The columns of $B(\T_i)$ are linearly dependent (e.g., their sum is zero), which contradicts $\det(B_E^V) \neq 0$. If $k = 1$, then $B_E^V$ has the form (\ref{eq:MST}), where now $B(\T_i)= 0$ is the $1 \times 1$ zero matrix. This again contradicts $\det(B_E^V) \neq 0$.  
\end{proof}

Now we complete the proofs of Theorem \ref{thm:kirchoff} and Remark \ref{rem:sign}. The matrix $L^V_W$ obtained from the Laplacian by row/column-deletion can be expressed in terms of the incidence matrix as
$$L^V_W= (B^{W})^\dagger B^{V},$$
where the dagger denotes transpose. By the Cauchy--Binet formula, the determinant is given by 
\begin{equation}\label{eq:CB}
\det(L_W^V)   = \sumd{E} \det(B_E^W )\det(B^V_E),
\end{equation}
where the sum is over all subsets $E \subseteq \E$ with complement of size $r$. By Lemma \ref{lem:kirchoff}, the $E$ summand $\det(B^W_E) \det(B^V_E)$ is non-zero exactly when the edges of $E$ form a spanning $\vert V \vert$-forest with the property that each tree contains exactly one vertex of $V$ and exactly one vertex of $W$. 

From Lemma \ref{lem:kirchoff}, we know that $\det(B^W_E) \det(B^V_E)$ is $1$ or $-1$. To finish the proof of Theorem \ref{thm:kirchoff}, we need to show that this sign is independent of $E$. To see this, fix $E$ and write
$$V = \left\{v_{i_1}, \ldots, v_{i_{\vert V \vert}} \right\}, \indent W = \left\{v_{j_1}, \ldots, v_{j_{\vert V \vert}} \right\}.$$
We have just seen that each tree determined by $E$ contains exactly one vertex from $V$ and one from $W$. It follows that $E$ determines a bijection 
$$t: \left\{i_1, \ldots, i_{\vert V \vert} \right\} \longrightarrow \left\{j_1, \ldots, j_{\vert V \vert} \right\}$$
between the indexing sets of $V$ and $W$ (note that this bijection is necessarily the identity on the intersection of these indexing sets). By Remark \ref{rem:1}, it follows that 
$$\det(B_E^W )\det(B^V_E) = (-1)^{i_1 + t(i_1)}(-1)^{i_2 + t(i_2)} \ldots (-1)^{i_{\vert V \vert} + t(i_{\vert V \vert})}.$$
Since $t$ is a bijection, the right-hand side is just
\begin{equation}\label{eq:sign}
(-1)^{i_1 + \ldots + i_{\vert V \vert} + j_1 + \ldots + j_{\vert V \vert}},
\end{equation}
which is independent of $t$, and hence of $E$. This finishes the proof of Theorem \ref{thm:kirchoff}. 

To prove the assertions of Remark \ref{rem:sign}, first assume that $W = V$. Then $\det(B_E^W) \det(B_E^V) = \det(B_E^V)^2$, and so (\ref{eq:CB}) is a sum of non-negative numbers. The identity (\ref{eq:VV}) follows. For (\ref{eq:ij}), we have $V = \left\{v_0, v_j\right\}$ and $W = \left\{v_0, v_i\right\}$, and so (\ref{eq:sign}) simplifies to $(-1)^{i+j}$. \qed

\section{The $G$-odometer}\label{sec:TheOdometer}

Let $G$ be a subgroup of $(\bb{R}, +)$, and we assume $\bb{Z} \subseteq G$. The standard inequality on $\bb{R}$ induces a partial ordering $\leq$ on $G^{\V'}$. Given $\sigma \in \bb{Z}^{\V'} \subseteq G^{\V'}$, we will be interested in solving the following inequality system for $u$:
\begin{eqnarray}
\sigma - L'  u & \leq & d-1 \label{eq:odo1}\\
u & \geq & 0\label{eq:odo2}
\end{eqnarray}
The next proposition provides existence and uniqueness for minimizers of the above system.

\begin{proposition}[Least Action Principle]\label{prop:existence}
Assume that $G$ is topologically closed as a subset of $\bb{R}$ and fix $\sigma \in \bb{Z}^{\V'}$. Then there is a unique $u_{\sigma}^G \in G^{\V'}$ satisfying (\ref{eq:odo1}) and (\ref{eq:odo2}) that is minimal in the sense that $u_{\sigma}^G \leq u$ for all $u \in G^{\V'}$ satisfying (\ref{eq:odo1}) and (\ref{eq:odo2}). 
\end{proposition}

We prove this in Section \ref{sec:proofofpropexistence} after we give several definitions and examples. The sandpile $u_{\sigma}^{G}$ from the proposition will be called the \emph{$G$-odometer of $\sigma$}. Since $\bb{Z}^{\V'} \subseteq G^{\V'}$, we automatically have
\begin{equation}\label{eq:odododo}
u_\sigma^{G} \leq u_{\sigma}^{\bb{Z}}.
\end{equation}

Consider the case $G = \bb{Z}$. The $\bb{Z}$-odometer $u_{\sigma}^{\bb{Z}}$ defined here is the same as the one discussed in the introduction; see \cite{LP,FLP,CP}. Moreover, the stabilization $\tau_\sigma$ of $\sigma$ from the introduction is given by $\tau_\sigma = \sigma - L' u_{\sigma}^{\bb{Z}}$.

We will say that $\sigma$ is \emph{$G$-immutable} if $u_\sigma^G = u_\sigma^{\bb{Z}}$. When $G = \bb{R}$ and there is little room for confusion, we will drop the $\bb{R}$ and simply say ``immutable'' in place of ``$\bb{R}$-immutable''. We will also say that a sandpile is \emph{mutable} if it is not immutable. 

\begin{remark}\label{rem:conventions}
In much of the sandpile literature (e.g., \cite{LP2,LMPU,LP3}), the term ``divisible sandpile'' refers to real-valued maps $\sigma: \V' \rightarrow \bb{R}$. These objects are generally studied in a framework that is effectively separate from standard (integer-valued) sandpiles. This allows the freedom to consider the system
\begin{equation}\label{eq:othersodo}
\begin{array}{rcl}
\sigma - L'  u & \leq & 1
\end{array}
\end{equation}
for $u: \V' \rightarrow \left[0, \infty \right)$, which is a rescaled version of (\ref{eq:odo1}) that is often convenient. However, our perspective in this paper is comparative in nature, in the sense that we want to simultaneously consider real-valued and integer-valued labelings. From our perspective, the rescaling leading to (\ref{eq:othersodo}) is artificial, and so we consider (\ref{eq:odo1}) even for real-valued maps. 
\hfill $\Diamond$ \end{remark}

The next section records several useful properties of immutability, and gives a few examples. In Section \ref{sec:proofofpropexistence}, we give a proof of the above Least Action Principle.

\subsection{Basic properties about $G$-immutable sandpiles}\label{sec:BasicProperties}

The following criterion follows directly from the definitions, but is an observation we will use repeatedly.

\begin{lemma}\label{lem:1}
A sandpile $\sigma$ is $G$-immutable if and only if the $G$-odometer $u_{\sigma}^{G}$ is an element of the subgroup $\bb{Z}^{\V'} \subseteq G^{\V'}$.
\end{lemma}

\begin{example} \label{ex:stable}
Here we show that every stable sandpile is $G$-immutable. If $\sigma$ is stable then $\sigma \leq d -1$, so $u = 0$ satisfies (\ref{eq:odo1}) and (\ref{eq:odo2}). This is clearly the minimal such solution over $G$, so $u_{\sigma}^{G} = 0$. The zero sandpile is integral, so $u_{\sigma}^{\bb{Z}} = u_{\sigma}^{G} $ by the previous lemma. 
\hfill $\Diamond$ \end{example}

Since $G \subseteq \bb{R}$, it follows that $u_\sigma^G \geq u_\sigma^{\bb{R}}$. This implies that $\sigma$ is $G$-immutable whenever $\sigma$ is $\bb{R}$-immutable. In fact, a partial converse of this holds: Assume that $G$ contains all components of $(L')^{-1}$ (e.g., $G = \det(L')^{-1} \bb{Z}$). Once we have Theorem \ref{thm:1} in hand, it will then follow that $u_\sigma^G = u_\sigma^{\bb{R}} $ for all uniformly large $\sigma$. Thus we have the following.

\begin{corollary}[Corollary to Theorem \ref{thm:1}]\label{cor:CorToThm}
Assume $L'$ is invertible over $G$ and $\sigma$ is uniformly large. Then $\sigma$ is $G$-immutable if and only if $\sigma$ is $\bb{R}$-immutable. 
\end{corollary}

The usefulness of the above corollary is that, to determine immutability for uniformly large sandpiles, one can work in, e.g., the cyclic group $\det(L')^{-1} \bb{Z}$ as opposed to $\bb{R}$. In the absence of a result such as Theorem \ref{thm:1}, this is not a priori clear from the definitions.

\begin{example}\label{ex:multik2}
Let $\Gamma$ be the connected multigraph with two vertices, $k \geq 1$ edges, and no self-loops; see Figure \ref{figure2}. As a warm-up for our computations below, we will classify the $(\bb{R}$-)immutable sandpiles on $\Gamma$. Fix a vertex to be the sink, and let $v$ be the other vertex. We claim that a sandpile $\sigma$ is immutable if and only if $\sigma$ is unstable or $k$ divides $\sigma(v) +1$. (Note that $k = d(v)$ is the degree of $v$.)

By Example \ref{ex:stable}, it suffices to assume $\sigma$ is unstable. Since there is only one non-sink vertex, the $\bb{R}$-odometer is the smallest real number $u(v) \geq 0$ so that
$$u(v) \geq (\sigma(v) + 1 - k)/k.$$
Clearly the minimum is $u_{\sigma}^{\bb{R}}(v) = (\sigma(v) + 1 - k) / k$, which is positive since $\sigma$ is unstable. The claim now follows from this expression and Lemma \ref{lem:1}. 

If $1/k \in G$, this shows that $G$-immutability is equivalent to $\bb{R}$-immutability for any sandpile (Corollary \ref{cor:CorToThm} only applies to uniformly large sandpiles).
\hfill $\Diamond$ \end{example}

\vspace{.5cm}
\begin{figure}[h]
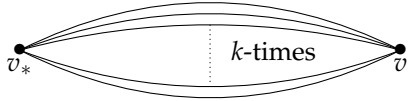

\centerline{\xy /r2pc/:,
(0,0)*={\xy
		(-3,0)="A0";
		(3,0)="A1";
		(0,.4)="T";
		(0,-.5)="B";
		"A0"*{\bullet};
		"A0"*+!U{\hbox{${{v_*}}$}};
		"A1"*{\bullet};
		"A1"*+!U{\hbox{${{v}}$}}; 
"T";"B" **\dir{.};
"A0";"A1" **\crv{(0,1.5)};
"A0";"A1" **\crv{(0,1.15)};
"A0";"A1" **\crv{(0,.8)};
"A0";"A1" **\crv{(0,-1.15)};
"A0";"A1" **\crv{(0,-1.5)};
(1,0)*{\textrm{$k$-times}};
\endxy};
\endxy}
\vspace{.5cm}
\caption{The multigraph of Example \ref{ex:multik2} with 2 vertices, and $k$ edges.\hspace{1.2cm} $\Diamond$}
\label{figure2}
\end{figure}

\subsection{Proof of Proposition \ref{prop:existence}}\label{sec:proofofpropexistence}

The proposition in the case $G = \bb{Z}$ is standard in the sandpile community; see \cite[Ch. 6]{CP} for a proof in the case of finite multigraphs $\Gamma$, as considered here. In particular, there exists a unique integral minimizer $u_\sigma^{\bb{Z}} \in \bb{Z}^{\V'}$ of (\ref{eq:odo1}) and (\ref{eq:odo2}).

For general $G$, define $u_{\sigma}^{G}: \V' \rightarrow G$ by
$$u_{\sigma}^{G} (v) \defeq \inf_u u(v)$$
where the infimum is over all $u \in G^{\V'}$ satisfying (\ref{eq:odo1}) and (\ref{eq:odo2}). Note that for each $v \in \V'$ this infimum is a well-defined element of $G$ because (i) $G$ is topologically closed in $\bb{R}$ and (ii) there does indeed exist an element of $G^{\V'}$ satisfying these inequalities---the $\bb{Z}$-odometer $u_\sigma^{\bb{Z}}$. We clearly also have that $u_{\sigma}^{G}$ satisfies (\ref{eq:odo2}), since inequalities are preserved by infima. We need to show that $u_{\sigma}^{G}$ satisfies (\ref{eq:odo1}); once this has been done, it will follow immediately from the definition that $u_{\sigma}^{G}$ is the \emph{unique} minimal solution. 

We begin with a preliminary computation: If $u_1, u_2$ satisfy (\ref{eq:odo1}) and (\ref{eq:odo2}), we will show that $u \defeq \min(u_1, u_2)$ does as well. Clearly it satisfies (\ref{eq:odo2}), so we turn to (\ref{eq:odo1}). To see this, fix a vertex $v \in \V'$. Without loss of generality, we may assume $u(v) = u_1(v) \leq u_2(v)$. Using $\sim$ to denote adjacency, we have
$$\begin{array}{rcl}
(L'u)(v) & = & d(v) u(v) - \sumd{w \sim v} \#(v, w) u(w) \\
& = & d(v) u_1(v) - \sumd{w \sim v} \#(v, w) u(w)\\
& \geq & d(v) u_1(v) - \sumd{w \sim v} \#(v, w) u_1(w)\\
& = & (L' u_1)(v)\\
& \geq & \sigma(v) - d(v) + 1.
\end{array}$$
Hence $u$ satisfies (\ref{eq:odo1}). 

With this in hand, for each vertex $v \in \V'$, fix a sequence $u_n^v \in G^{\V'}$ satisfying (\ref{eq:odo1}) and (\ref{eq:odo2}) with $u_n^v(v) \rightarrow u_{\sigma}^{G}(v)$. Define a new sequence $u_n'$ by $u_n' \defeq \min_v(u_n^v)$ where the minimum is taken over all $v \in \V'$. Clearly $u_n'$ satisfies (\ref{eq:odo2}). We also have 
$$\vert u_n'(v) - u_{\sigma}^{G}(v) \vert = u_n'(v) - u_{\sigma}^{G}(v) \leq u_n^v(v) - u_{\sigma}^{G}(v) = \vert u_n^v(v) - u_{\sigma}^{G}(v)  \vert \rightarrow 0$$
for all $v \in \V'$. This shows that $u_n'$ converges to $u_{\sigma}^{G}$ as elements of $G^{\V'}$ and so $L' u_n' \rightarrow L' u_{\sigma}^{G}$. The preliminary computation of the previous paragraph shows that $u_n'$ satisfies (\ref{eq:odo2}), so for each $v \in \V'$ we have
$$\sigma(v)  - d(v) + 1 \leq (L' u_n')(v).$$
Taking the limit in $n$ implies $\sigma(v) - d(v) + 1 \leq (L' u_{\sigma}^{G})(v)$, which is (\ref{eq:odo1}). 
\qed

\section{Proofs of the main results}

\subsection{Proof and discussion of Theorem \ref{thm:1}}\label{sec:ProofOfThm1.1}

We begin by proving Theorem \ref{thm:1} from the introduction. We then consider three examples that show the uniformly large hypothesis (ULH) cannot be removed entirely. 

\begin{proof}[Proof of Theorem \ref{thm:1}]
Fix a sandpile $\sigma$. Since $L'$ is invertible over $\bb{R}$, the equation 
\begin{equation} \label{eq 3}
\sigma - L'\tilde{u} = d-1
\end{equation}
has a unique solution $\tilde{u} \in \bb{R}^{\V'}$. When $\sigma$ is uniformly large, we have $\sigma - d+1 \geq 0$. By Remark \ref{Rem:L'inverse}, the entries of $L'$ are all non-negative, so $\tilde{u} = (L')^{-1}(\sigma - d+1) \geq 0$. This and (\ref{eq 3}) imply that $\tilde{u} \in \bb{R}^{\V'}$ satisfies (\ref{eq:odo1},\ref{eq:odo2}). In particular, since $u_\sigma^{\bb{R}}$ is minimal among all such solutions, we have $q \defeq \tilde{u} - u_\sigma^{\bb{R}} \geq 0$. Note that the definition of $\tilde{u}$ combines with (\ref{eq:odo1}) for $u_\sigma^{\bb{R}}$ to give
$$L'q = \sigma - d + 1 - L' u_\sigma^{\bb{R}} \leq 0.$$
Thus $\langle L'q, q \rangle \leq 0$, since $q \geq 0$. On the other hand, $L'$ is positive definite, so $0 \leq \langle L' q, q \rangle \leq 0$ and hence $q = 0$. This gives $u_\sigma^{\bb{R}} = (L')^{-1}(\sigma - d+1)$. The remaining assertion of the theorem follows from Lemma \ref{lem:1}.  
\end{proof}

The following examples show that none of the claims of the theorem hold with the ULH dropped entirely. We begin with a very simple example.

\begin{example}\label{ex:trivial}
Consider the sandpile $\sigma = 0$. This is stable, so $u_\sigma^{\bb{Z}} = u_\sigma^{\bb{R}} = 0$, which implies it is immutable. On the other hand, $(L')^{-1}(\sigma - d+1) = (L')^{-1}(-d+1)$. If $\Gamma$ is any graph with $d \neq 1$, then $(L')^{-1}(-d+1) \neq u_\sigma^{\bb{R}}$; this means the first conclusion of Theorem \ref{thm:1} would fail with the ULH dropped. If $\Gamma$ is such that $(L')^{-1}(-d+1)$ is not integral (e.g., the multigraph of Example \ref{ex:multik2} for $k \geq 2$), then this also means the second conclusion of the theorem would fail without the ULH. 
\hfill $\Diamond$ \end{example}

The above example is not entirely satisfying, since it still leaves room for the ULH to be replaced by something significantly simpler such as ``$\sigma$ is not 0'', or even ``$\sigma$ is not stable''. The following example shows that neither of the conclusions of Theorem \ref{thm:1} would hold if the ULH were replaced by even the stronger of these two assumptions: that $\sigma$ is not stable.

\begin{example}\label{ex:complete1}
Let $\Gamma  = K_3$ be the complete graph on 3 vertices. Order the non-sink vertices to identify $\bb{R}^{\V'} \cong \bb{R}^2$, and consider the sandpile $\sigma = (2, 0)^\dagger$; the dagger denotes the transpose. Note that this is not uniformly large, and it is not stable. One can compute directly that $(L')^{-1}(\sigma - d+1)=(1/3,-1/3)^\dagger$. In particular, $(L')^{-1}(\sigma - d+1)$ is not equal to $u_\sigma^{\bb{R}}$, since the $\bb{R}$-odometer is required to be non-negative. 

Next, we show that $\sigma$ is mutable. Indeed, we will show
$$u_\sigma^{\bb{Z}} = \left(\begin{array}{c}
													1\\
													0
													\end{array}\right) , \indent u_\sigma^{\bb{R}} = \left(\begin{array}{c}
													1/2\\
													0
													\end{array}\right).$$	
For this, note that the system (\ref{eq:odo1},\ref{eq:odo2}) is equivalent to 
\begin{equation}\label{eq:systemisdown}
1 \leq 2x -y, \indent 1 \geq x - 2y, \indent x \geq 0 , \indent y \geq 0.
\end{equation}
where we have written $u = (x, y)^\dagger$ in coordinates. From here, it is easy to see that the claimed values of $u_\sigma^{\bb{Z}}$ and $u_\sigma^{\bb{R}}$ satisfy (\ref{eq:systemisdown}) and hence (\ref{eq:odo1}, \ref{eq:odo2}). It therefore suffices to show that these are minimal in $\bb{Z}$ and $\bb{R}$, respectively. For $u_\sigma^{\bb{Z}}$, minimality is easy since the only nonnegative vector that is smaller is $(0, 0)^\dagger$, and this does not satisfy (\ref{eq:systemisdown}). That $u_\sigma^{\bb{R}}$ is minimal follows from the first and last inequalities in (\ref{eq:systemisdown}), which imply $x \geq 1/2$. 
\hfill $\Diamond$ \end{example}

The above example still leaves open the possibility that the integrality of $(L')^{-1}(\sigma - d + 1)$ could detect immutability in the absence of the ULH. The following example shows this is not the case. 

\begin{example}\label{ex:complete2}
Let $\Gamma  = K_4$ be the complete graph on 4 vertices, and identify $\bb{R}^{\V'} \cong \bb{R}^3$ by ordering the vertices. Consider here the sandpile $\sigma = (4,0, 0)^\dagger$. As in the previous example, one can check directly that
$$u_\sigma^{\bb{Z}} = \left(\begin{array}{c}
													1\\
													0\\
													0
													\end{array}\right) , \indent u_\sigma^{\bb{R}} = \left(\begin{array}{c}
													1/2\\
													0\\
													0
													\end{array}\right),$$
													so $\sigma$ is mutable. On the other hand,
													$$(L')^{-1}(\sigma - d+1) = 	\left(\begin{array}{c}
													0\\
													-1\\
													-1
													\end{array}\right)$$
													which is integral. 
\hfill $\Diamond$ \end{example}

For similar phenomena, but on a path, see Example \ref{ex:path4}.

\subsection{Proof and discussion of Corollary \ref{cor:2}}\label{sec:Proof3}

Throughout this section, we assume that $\Gamma$ is the cone of a regular graph and the sink is the cone point. We begin with two observations that are special to this setting, and then use these to prove Corollary \ref{cor:2}. At the end, we discuss the extent to which the various hypotheses/conclusions of the corollary can be weakened/strengthened.

\begin{lemma}\label{lem:2}
View the degree map $d$ as a sandpile on $\Gamma$. Then
$$d-1 = L'(d-1).$$
\end{lemma}

\begin{proof}
The assumption that $\Gamma$ is a cone with sink given by the cone point implies that $L' c = c$ for any constant function $c \in \bb{R}^{\V'}$. Since $\Gamma$ is the cone of a \emph{regular} graph, it follows that $d$ is constant on the non-sink vertices. Thus $d -1$ is a constant sandpile, so the lemma follows. 
\end{proof}

\begin{lemma}\label{lem:3}
Fix a sandpile $\sigma$ and write $\sigma = L' a$ for some $a \in \bb{R}^{\V'}$ satisfying $a \geq d-1$. Then $u_{\sigma}^{\bb{R}} = a - d + 1$.
\end{lemma}

\begin{proof}
The assumptions on $a$ combine with Lemma \ref{lem:2} to give
$$\sigma - L' (a - (d-1)) =  d-1, \indent a - (d-1) \geq 0,$$
and so $u =a - (d-1)$ satisfies (\ref{eq:odo1}, \ref{eq:odo2}). The minimality of $u_\sigma^{\bb{R}}$ then gives $u_\sigma^{\bb{R}} \leq a - (d-1)$. For the reverse inequality, notice that $\sigma + (d-1)$ is uniformly large since $\sigma$ is a sandpile. Theorem \ref{thm:1} implies its $\bb{R}$-odometer is given by
$$u^{\bb{R}}_{\sigma +(d-1)} = (L')^{-1}(\sigma + (d-1) - (d-1)) = a .$$
By Lemma \ref{lem:2} again, we have
$$\sigma + (d-1) - L'(u_\sigma^{\bb{R}} + (d-1)) = \sigma - L' u_\sigma^{\bb{R}} \leq d- 1,$$
Thus $u = u_\sigma^{\bb{R}} + (d-1)$ satisfies (\ref{eq:odo1}, \ref{eq:odo2}) relative to the sandpile $\sigma + (d-1)$. By minimality of the $\bb{R}$-odometer for $\sigma + (d-1)$, we therefore have
$$a = u^{\bb{R}}_{\sigma +(d-1)} \leq u_\sigma^{\bb{R}} + (d-1),$$
which establishes the reverse inequality we were after.
\end{proof}

\begin{proof}[Proof of Corollary \ref{cor:2}] 
First assume $\sigma$ is a uniformly large sandpile. By Theorem \ref{thm:1}, we have that $\sigma$ is immutable if and only if $u_\sigma^{\bb{R}} = (L')^{-1}(\sigma - d+1)$ is integral. Lemma \ref{lem:2} implies $(L')^{-1}(\sigma - d+1) = (L')^{-1}(\sigma) - (d-1)$, so $\sigma$ is immutable if and only if $a\defeq (L')^{-1}\sigma$ is integral. This also shows that $a = u_\sigma^{\bb{R}} + (d-1) \geq d-1$, so $a$ is uniformly large. In particular, this proves (a).

Now we prove (b), so we assume $\sigma = L' a \geq 0$ for some uniformly large sandpile $a$. By Lemma \ref{lem:3}, we have $u_\sigma^{\bb{R}} = a + d-1$. Since $a$ is integral, the immutability of $\sigma$ is immediate from Lemma \ref{lem:1}. 
\end{proof}

We end this section with a discussion of the various hypotheses and conclusions in the statement of Corollary \ref{cor:2}. Our first example shows that we cannot necessarily conclude in Corollary \ref{cor:2} (b) that $\sigma$ is uniformly large. Said differently, this corollary produces a class of immutable sandpiles, and the following example shows that this class contains sandpiles that are neither stable, nor uniformly large. 

\begin{example}\label{ex:expanded}
Let $\Gamma = K_3$ be as in Example \ref{ex:complete1}, and consider the uniformly large sandpile $a = (1, 2)^\dagger$. Then $\sigma = L' a = (0, 3)^\dagger$ is an unstable sandpile that is not uniformly large. By Corollary \ref{cor:2}, $\sigma$ is immutable. 
\hfill $\Diamond$ \end{example}

Our remaining examples all address the necessity of the hypotheses of Corollary \ref{cor:2}. We begin with the hypothesis in (a) that $\sigma$ is uniformly large. We give two examples showing that this hypothesis cannot be removed entirely. The first of these examples is a refinement of Example \ref{ex:trivial}.

\begin{example}\label{ex:ex}
Let $\Gamma$ be a graph with non-trivial critical group $K(\Gamma)$. Then there are stable sandpiles not in the image of $L': \bb{Z}^{\V'} \rightarrow \bb{Z}^{\V'}$; by Corollary \ref{cor:2} (a), these are not uniformly large. 

To be concrete, consider $\Gamma = K_3$, and let $\sigma = (1, 0)^{\dagger}$. Then $\sigma$ is stable and so immutable. However, $\sigma$ is not in the image of $L'$ on $\bb{Z}^{\V'}$ since $(L')^{-1} \sigma = (2/3, 1/3)^\dagger \notin \bb{Z}^{\V'}$. 
\hfill $\Diamond$ \end{example}

We view stable sandpiles as ``trivially immutable''. As such, it would be more fulfilling to have an example similar to that of Example \ref{ex:ex}, but with $\sigma$ unstable. This is supplied by the following.

\begin{example}\label{ex:thm2.2}
Let $\Gamma = K_3$ and consider the sandpile $\sigma = (4, 0)^\dagger$. This is not uniformly large and not stable. As in Example \ref{ex:complete1}, one can show that $u_{\sigma}^{\bb{Z}} = u_{\sigma}^{\bb{R}}   = (2, 1)^{\dagger}$ and so $\sigma$ is immutable. However, $\sigma$ is not in the image of $L': \bb{Z}^{\V'} \rightarrow \bb{Z}^{\V'}$ since $(L')^{-1} \sigma = (8/3, 4/3)^\dagger$. 
\hfill $\Diamond$ \end{example}

Now we turn to the necessity of the hypotheses of Corollary \ref{cor:2} (b). The following example is a little silly, but shows we do not get the condition $\sigma \geq 0$ for free, simply by assuming $\sigma = L'a$ for a uniformly large $a$. 

\begin{example}\label{ex:thm2}
Let $\Gamma = K_3$ and consider the uniformly large sandpile $a = (1,3)^\dagger$. Then $\sigma = L' a = (-1 , 5)^\dagger$ is not non-negative, and so not a sandpile.
\hfill $\Diamond$ \end{example}

Our last example illustrates that the hypothesis that $a$ is uniformly large cannot be removed entirely.

\begin{example}
Consider $\Gamma = K_4$ and $a = (2, 1, 1)^\dagger$, which is not uniformly large. Then $\sigma = L' a = (4, 0, 0)^\dagger \geq 0$. However, as we saw in Example \ref{ex:complete2}, $\sigma$ is mutable. 
\hfill $\Diamond$ \end{example}

\subsection{Proof of Corollary \ref{cor:1}}\label{sec:ProofOfCor1.1}

Let $v$ be as in the statement of the corollary. Define a sandpile $\sigma$ by $\sigma(v) = d(v)$ and $\sigma(w) = d(w)-1$ for $w \neq v$. This is clearly uniformly large. By Theorem \ref{thm:1}, to see that $\sigma$ is mutable, it suffices to show that $(L')^{-1}(\sigma - d+1) \in \bb{R}^{\V'}$ has a non-integral component. To do this, we will show that the $v$-component of the vector $(L')^{-1}(\sigma - d+1) $ lies in the interval $(0, 1)$.

Note that by the cofactor expansion for the inverse of $L'$, this $v$-component is given by $  L'_{v,v} / \det(L')$, where $L'_{v, v}$ is the minor of $L'$ obtained by deleting the column and row corresponding to $v$. It follows from Lemma \ref{cor:kirchoff} that $ \det(L') = \vert \Ss_1 \vert$ is the number of spanning trees of $\Gamma$; see Notation \ref{def:spanningsets}. Similarly, by Corollary \ref{cor:L'inverse}, the minor $L'_{v, v} = \vert \Ss_2(v, v)\vert $ is the number of spanning 2-forests in $\Gamma$, with one tree containing $v_*$ and the other containing $v$. We will show that (i) $\Ss_2(v, v)$ is not empty, so $L'_{v,v} / \det(L') >0$, and (ii) there is an injection $\F: \Ss_2(v, v) \hookrightarrow \Ss_1$ that is not surjective, so $ L'_{v,v} / \det(L')  < 1$; the corollary will follow. 

We will first show (i). Let $\mathcal{T}_1$ be the tree consisting of exactly the one vertex $v$. By the assumption (a) in the statement of the corollary, there is a disjoint tree $\mathcal{T}_0$ containing all vertices other than $v$. Then $( \T_0, \T_1) \in \Ss_2(v,v)$.

For (ii), recall we have assumed that $v_*$ is adjacent to $v$. Fix one edge $e$ that is incident to both $v$ and $v_*$. Now let $ ( \T_0, \T_1 ) \in \Ss_2(v, v)$ and assume $v_*$ is contained in $\T_0$. Note that neither $\T_0$ nor $\T_1$ contains $e$. Define $\F(( \T_0, \T_1 ))$ to be the subgraph of $\Gamma$ obtained from $\T_0 \cup \T_1$ by including this edge $e$. It is not hard to see that $\F(( \T_0, \T_1 ))$ is a tree, and this is necessarily a spanning tree since all vertices of $\Gamma$ are contained in it by construction. Thus $\F$ is well-defined, and its injectivity follows readily from the construction. Also, every spanning tree in the image of $\F$ contains the edge $e$. The failure of $\F$ to be surjective will therefore follow from the next claim.

\medskip

\noindent \emph{Claim}: There is a spanning tree of $\Gamma$ that does not contain the edge $e$. 

\medskip

We have assumed in (b) in the statement of the corollary that the degree of $v$ is at least 2. This implies there is another edge $e'$ incident to $v$. If $e'$ is also incident to $v_*$, then the claim follows by taking any spanning tree in the image of $\F$, deleting $e$ and replacing it by $e'$. We may therefore assume that $e'$ is incident to some other vertex $w \neq v, v_*$. Recall from (i) that there is a tree $\T$ in $\Gamma$ that contains all vertices except $v^*$. In particular, $\T$ contains $w$ and it does not contains $e$ (nor any other edges adjacent to $v$). Then adding the edge $e'$ to $\T$ produces the desired spanning tree. \qed

\section{Examples}\label{sec:Examples}

Here we explore Theorem \ref{thm:1} through the lens of three classes of graphs. We begin with trees. In this case, we find that all uniformly large sandpiles are immutable. Though the analysis for trees is particularly simple, the conclusion is instructive. We then move on to complete graphs and wheel graphs. In these examples, we provide a direct classification of uniformly large immutable sandpiles in terms of systems of linear equations. Moreover, the conditions of these classifications can be checked directly from the sandpile (e.g., without having to invert a matrix).

\subsection{Trees}\label{sec:Trees}

Assume that $\Gamma$ is a connected tree and choose any vertex as the sink.

\begin{theorem} \label{indiv path}
All uniformly large sandpiles are immutable. 
\end{theorem}

\begin{proof}
By Kirchoff's Matrix Tree Theorem, we have $\det(L') =  1$ and so $L'$ is invertible over $\bb{Z}$. Thus $(L')^{-1}(\sigma - d+1) \in \bb{Z}^{\V'}$ is integral for all $\sigma \in \bb{Z}^{\V'}$. If $\sigma$ is uniformly large, then Theorem \ref{thm:1} implies $u_{\sigma}^{\bb{R}} =  (L')^{-1}(\sigma - d+1)$ and so $\sigma$ is immutable.
\end{proof}

\begin{example}\label{ex:immutablepath}
Let $\Gamma = P_3$ be a path with $3$ vertices, labeled as in Example \ref{path irl}. We claim that all sandpiles on $P_3$ are immutable. By Theorem \ref{indiv path} and Example \ref{ex:stable}, it suffices to verify the claim for those sandpiles that are neither stable, nor uniformly large. The only sandpiles of this type are of the form $\sigma = (0, k)^\dagger$ for $k \geq 1$. Then (\ref{eq:odo1},\ref{eq:odo2}) become
$$-2x + y \leq 1, \indent k + x - y \leq 0, \indent x, y \geq 0$$
for $u = (x, y)^\dagger$. Combining the first two, we have $k+x \leq y \leq 1 + 2x$, which implies $x \geq k -1$, and so $y \geq 2k - 1$. It follows that $u_\sigma^{\bb{R}} = (k-1, 2k-1)^{\dagger}$ is the minimal such solution. This is integral, and so $\sigma$ is immutable. 
\hfill $\Diamond$  \end{example}

Despite Examples \ref{ex:multik2} (with $k = 1$) and \ref{ex:immutablepath}, there are paths that admit mutable sandpiles, as the next example shows.

\begin{example}\label{ex:path4}
Consider the case where $\Gamma = P_4$ is the path with 4 vertices, labeled as before. Consider the sandpile $\sigma = (2, 0, 0)^\dagger$. Arguing as in the previous example, one can show that $u_\sigma^{\bb{R}} = (1/2, 0, 0)^{\dagger}$, while $u_\sigma^{\bb{Z}} = (1, 0, 0)^\dagger$. In particular, $\sigma$ is mutable. 

We note also that $(L')^{-1}(\sigma - d+1)$ is integral and $u_\sigma^{\bb{R}} \neq (L')^{-1}(\sigma - d+1)$ (which can be easily verified from the formula in Example \ref{path irl}). This therefore provides another example where the conclusions of Theorem \ref{thm:1} fail in the absence of the uniformly large hypothesis.
\hfill $\Diamond$ \end{example}

	\subsection{Complete Graphs}\label{sec:CompleteGraphs}
	
Let $\Gamma = K_{n+1}$ be the complete graph on $n+1 \geq 3$ vertices, and fix any vertex to be the sink $v_*$. Label the remaining vertices $v_1, \ldots, v_n$ and thus identify $\bb{R}^{\V'} \cong \bb{R}^n$. For $\sigma \in \bb{R}^{\V'}$, write $(\sigma_1, \ldots, \sigma_n)$ for its components under this identification.

\begin{theorem} \label{thm complete}
Suppose $\sigma \in \bb{Z}^{\V'}$. Then $\sigma$ is in the image of $L': \bb{Z}^{\V'} \rightarrow \bb{Z}^{\V'}$ if and only if 
\begin{equation} \label{eq:completemodn+1} \sigma_i - \sigma_j  \equiv  0 \text{ mod } n+1 , \indent \forall 1 \leq i, j \leq n.
\end{equation}
In particular, if $\sigma$ is uniformly large, then $\sigma$ is immutable if and only if (\ref{eq:completemodn+1}) holds.
\end{theorem}

We begin by computing the inverse of $L'$. 

\begin{lemma} \label{complete irl}
View $L'$ as a matrix relative to the identification $\bb{R}^{\V'} \cong \bb{R}^n$. Then
\[ (L')^{-1} = \frac{1}{n+1} \begin{pmatrix}
		2 & 1 & \dots & 1 \\
		1 & 2 & \dots & 1 \\
		\vdots & \vdots & \ddots & \vdots \\
		1 & 1 & \dots & 2
	\end{pmatrix} \]
\end{lemma}

\begin{proof}[Proof of Lemma \ref{complete irl}]
Cayley's formula tells us that $\det(L') =  (n+1)^{n-1}$. By Corollary \ref{cor:L'inverse}, to determine the $(i, j)$-component $(L')^{-1}_{(i,j)}$, it therefore suffices to determine the size of $\Ss_2(v_i, v_j)$.

First assume $i = j$. Given an edge $e$ in $K_{n+1}$, let $n_e$ be the number of spanning trees that contain $e$. There is a one-to-one correspondence between $\Ss_2(v_i, v_i)$ and the set of spanning trees that contain the edge $e(i)$ connecting $v_0$ and $v_i$; that is $\vert \Ss_2(v_i, v_i) \vert = n_{e(i)}$. On the other hand, the symmetries of $K_{n+1}$ imply that $n_e$ is independent of $e$, so $\vert \Ss_2(v_i, v_i) \vert = n_e$ for all $e$. To obtain a formula for $n_e$, we will count the total number of edges in all spanning trees in two different ways. First, there are $n$ edges in each spanning tree, and Cayley's formula says that there are $(n+1)^{n-1}$ spanning trees in total. So the total number of all edges in all spanning trees is $n (n+1)^{n-1}$. Second, this total number can be counted by summing over $n_e$ as $e$ ranges over all edges. There are $n+1$ choose $2$ such edges, and since $n_e$ is independent of $e$, this combines with our first counting strategy to give
$$n (n+1)^{n-1}= \textrm{total number of edges in all spanning trees} = {n+1 \choose 2} \vert \Ss_2(v_i, v_i) \vert.$$
In summary, 
$$(L')^{-1}_{ii} = \vert \Ss_2(v_i, v_i) \vert  / \det(L') = 2 (n+1)^{n-2} / (n+1)^{n-1} = 2 / (n+1).$$

Now consider the case $i \neq j$. We will show there is a 2-1 correspondence from $\Ss_2(v_i, v_i)$ to $\Ss_2(v_i, v_j)$. For this, fix $\left( \T_0, \T_1 \right) \in \Ss_2(v_i, v_i)$, so $v_* \in \T_0$ and $v_i \in \T_1$. Then either $v_j \in \T_0$ or $v_j \in \T_1$, but both of these cannot happen simultaneously. By symmetry of the complete graph $\Gamma$, there is a one-to-one correspondence between those $\left( \T_0, \T_1 \right) \in \Ss_2(v_i, v_i)$ with $v_j \in \T_0$ and those with $v_j \in \T_1$. Thus $\vert \Ss_2(v_i, v_i) \vert = 2\vert \Ss_2(v_i, v_j) \vert$. It then follows from the previous case that 
$$(L')^{-1}_{ij} = \vert \Ss_2(v_i, v_j) \vert  / \det(L') =  1 / (n+1).$$ 
\end{proof}

\begin{proof}[Proof of Theorem \ref{thm complete}]
Note that $\Gamma = K_{n+1}$ is the cone of the regular graph $K_n$. We will show that the image of $L'$ is characterized by (\ref{eq:completemodn+1}); the remaining claim of Theorem \ref{thm complete} will then be immediate from Corollary \ref{cor:2}.

First assume $\sigma$ satisfies (\ref{eq:completemodn+1}). Then there exists $k \in \bb{Z}$ so that $\sigma_{i} = k \text{ mod } n+1$ for all $1 \leq i \leq n$. Let $a = (L')^{-1}(\sigma) \in \bb{R}^{\V'}$ and write $a_i =a(v_i)$ for the $i$th component; we need to show that the $a_i$ are integers. By Lemma \ref{complete irl} we have
\begin{equation}\label{eq:ui}
(n+1)a_{i}  =  2\sigma_{i} + \sumd{j \neq i} \sigma_{j}  =  \sigma_{i} + \sumd{j } \sigma_{j}  .
\end{equation}
Temporarily working mod $n+1$, our hypothesis on $\sigma$ gives 
$$\sigma_{i} + \sum_{j} \sigma_{j} \equiv k + (k n) \equiv  0\indent \mod n+1.$$ 
Thus $n+1$ divides $\sigma_{i} + \sum_{j } \sigma_{j} $ and so it follows from (\ref{eq:ui}) that $a_i \in \bb{Z}$, as desired. 

Conversely, assume that $\sigma = L' a$ for some $a \in \bb{Z}^{\V'}$. Note that
$$L' = \begin{pmatrix}
	n & -1 & \dots & -1 \\
	-1 & n & \dots & -1 \\
	\vdots & \vdots & \ddots & \vdots \\
	-1 & -1 & \dots & n
	\end{pmatrix}$$
and so
$$\sigma_{i} =  n a_{i} - \sumd{j \neq i} a_j =   (n+1) a_{i} - \sumd{j } a_j .$$
Thus $\sigma_{i} - \sigma_{j}  =  (n+1)(a_{i} - a_{j}) $. Since the $a_i$ are integers, this recovers (\ref{eq:completemodn+1}).
\end{proof}

	\subsection{Wheel Graphs}\label{sec:WheelGraphs}
	
	Let $\Gamma = W_{n+1}$ be a wheel graph on $n+1 \geq 4$ vertices. Choose the sink $v_*$ to be the central vertex, and label the boundary vertices $v_1, \ldots, v_n$ cyclically. We treat the index of the $v_i$ modulo $n$, so $i \in \bb{Z}/ n \bb{Z}$; this reflects the rotational symmetry of $W_{n+1}$. (In this section, we will therefore view $v_0 = v_n$, as opposed to our conventions of the previous sections where $v_0$ was the sink.)
	
	Given $\sigma \in \bb{R}^{\V'}$, we will write $\sigma_i \defeq \sigma(v_i)$. Write $F_k$ for the $k$th Fibonacci number, and $A_k$ for the $k$th Lucas number. Our main result is as follows. 
	
	\begin{theorem}\label{thm:wheel} 
	Let $\sigma \in \bb{Z}^{\V'}$. Then $\sigma$ is in the image of $L': \bb{Z}^{\V'} \rightarrow \bb{Z}^{\V'}$ if and only if the following holds for all $i \in \bb{Z}/n\bb{Z}$:
\begin{eqnarray} 
    \textrm{$n$ even:} & A_{n}\sigma_{i+n/2} + A_{0}\sigma_{i} + \sum_{m=1}^{\frac{n}{2}-1} A_{2m}(\sigma_{i+m}+\sigma_{i-m}) \equiv 0 \mod 5 F_n \label{eq:neven}\\
    \textrm{$n$ odd:} &     F_{n}\sigma_{i+(n+1)/2} + \sum_{m=1}^{\frac{n-1}{2}} F_{2m-1}(\sigma_{i+m}+\sigma_{i+1-m}) \equiv 0 \mod A_{n}. \label{eq:nodd}
\end{eqnarray}
	
	In particular, if $\sigma$ is uniformly large, then $\sigma$ is immutable if and only if (\ref{eq:neven},\ref{eq:nodd}) holds for all $i \in \bb{Z}/ n \bb{Z}$.
\end{theorem}

The proof of Theorem \ref{thm:wheel} is given in Section \ref{sec:WheelProofs}. As with complete graphs, our proof relies on an explicit formula for the inverse of the reduced Laplacian, which is computed in Section \ref{sec:CountingTrees}.

The system (\ref{eq:neven}) (resp. (\ref{eq:nodd})) consists of $n$ equations in the ring $\bb{Z}/5 F_n \bb{Z}$ (resp. $\bb{Z}/ A_n \bb{Z}$). Whenever one has a relatively prime factorization of $5F_n$ (resp. $A_n$), then the system (\ref{eq:neven}) (resp. (\ref{eq:nodd})) reduces in size. For example, if $n$ is even and $F_n$ and 5 are relatively prime, then the system (\ref{eq:neven}) is equivalent to the same system of congruences, but in the smaller rings $\bb{Z}/5 \bb{Z}$ and $\bb{Z}/ F_n \bb{Z}$, simultaneously. 

As another example, when $n = 2^k$ is a power of 2, we have $5F_{2^{k}}= 5\prod_{\ell=1}^{k-1} A_{2^{\ell}}$ (which can be seen by inducting on the identity $F_{2a} = A_a F_a$). This is a relatively prime factorization by \cite{McD}. As we show in Section \ref{sec:WheelProofs}, an extension of the argument of the previous paragraph produces the following.

	\begin{corollary} \label{cor:specialwheel}
Assume $n = 2^k$ for $k \geq 2$ and fix $\sigma \in \bb{Z}^{\V'}$. Then (\ref{eq:neven}) holds for all $i \in \bb{Z}/n \bb{Z}$ if and only if 
\begin{equation} \label{eq:mod5spw}
\sum_{m=0}^{2^{k}-1} (-1)^{m} \sigma_{m} \equiv  0 \mod 5
\end{equation}
and, for all $1 \leq \ell \leq k - 1$ and $0 \leq i \leq 2^{\ell} -1$,
\begin{equation}\label{eq:modAspw}
\sumdd{c=0}{2^{k-\ell}-1} (-1)^{c}\Big[A_{0}\sigma_{i + 2^{\ell}c} + \sumdd{j=1}{2^{\ell-1}-1}A_{2j}(\sigma_{i + 2^{\ell}c+j}+\sigma_{i + 2^{\ell}c-j})\Big]  \equiv  0 \mod A_{2^\ell}.
\end{equation}
\end{corollary}

As suggested above, the main utility of Corollary \ref{cor:specialwheel} over Theorem \ref{thm:wheel} is that the rings of the former are significantly smaller. We note also the restriction on the index $i$ in the corollary, which implies the systems expressed in (\ref{eq:mod5spw}) and (\ref{eq:modAspw}) contain a total of $n - 1 = 2^k -1$ equations (so there is one less equation than (\ref{eq:neven})). Another simplification expressed by the corollary is that the Lucas numbers appearing in (\ref{eq:mod5spw}) and (\ref{eq:modAspw}) have already been reduced mod $5$ and $A_{2^\ell}$, respectively. This reflects various Lucas number identities that are special for powers of 2; see Lemma \ref{lem:mod A conj}. We encourage the reader to write out the system (\ref{eq:neven}), and then (\ref{eq:mod5spw},\ref{eq:modAspw}) for $n = 2$ and $4$ to get a sense for the symmetries expressed by the latter system that are hidden in the former. 

Finally, we want to emphasize the helpfulness of experimental math, which we found to be an invaluable tool for identifying the results of this section. Indeed, the viability of the present wheel graph example began through various computations of immutability with the assistance of the computer algebra system Sage. When working with general (even) $n$, we initially observed no discernible pattern. However, when we specialized to $n = 2^k$, a pattern clearly emerged and we were able to conjecture the formulas of (\ref{eq:mod5spw},\ref{eq:modAspw}). After that, we wrote down a proof of the expression for the Laplacian given Theorem \ref{thm:laplacewheel}. We were once again aided by the computer, which suggested our initial formulas could be reduced significantly. Moreover, due to the overwhelming computational evidence, this led us to guess at the relative primality of $A_{2^\ell}$ for distinct $\ell \geq 1$. We are grateful to a Mathematics Stack Exchange answer by Peter Woolfitt for pointing out the reference \cite{McD}, where this relative primality is established.

\subsubsection{Counting trees and 2-forests}\label{sec:CountingTrees}

The main result of this section is the following explicit formula for the inverse of $L'$.

\begin{theorem}\label{thm:laplacewheel}
If $n$ is even, then
$$(L')^{-1} = \frac{1}{5 F_n} \begin{pmatrix}
                            A_{n} & A_{n-2} & A_{n-4} & \dots & A_2 & A_{0} & A_{2} & \dots & A_{n-2} \\
							A_{n-2} & A_{n} & A_{n-2} & \dots & A_4 & A_{2} & A_{0} & \dots & A_{n-4} \\
							A_{n-4} & A_{n-2} & A_{n} & \dots &A_6 & A_{4} & A_{2} & \dots & A_{n-6} \\
							\vdots & \vdots & \vdots & \ddots & \vdots & \vdots &\vdots& \ddots & \vdots \\
							A_2 & A_4 & A_6 & \ldots & A_{n} & A_{n-2} & A_{n-4} & \ldots & A_0 \\
							A_{0} & A_{2} & A_{4} & \dots & A_{n-2} & A_{n} & A_{n-2} & \dots & A_{2} \\
								A_{2} & A_{0} & A_{2} & \dots &A_{n-4} & A_{n-2} & A_{n} & \dots & A_{4} \\
							\vdots & \vdots & \vdots & \ddots & \vdots & \vdots &\vdots& \ddots & \vdots \\
							A_{n-2} & A_{n-4} & A_{n-6} & \dots &A_0 & A_{2} & A_{4} & \dots & A_{n}
							\end{pmatrix}.$$
							
If $n$ is odd, then
$$(L')^{-1} = \frac{1}{A_n} \begin{pmatrix}
							F_{n} & F_{n-2} & F_{n-4} & \dots & F_3 & F_{1} & F_{1} & \dots & F_{n-2} \\
							F_{n-2} & F_{n} & F_{n-2} & \dots &F_5 & F_{3} & F_{1} & \dots & F_{n-4} \\
							F_{n-4} & F_{n-2} & F_{n} & \dots & F_7 &F_{5} & F_{3} & \dots & F_{n-6} \\
							\vdots & \vdots & \vdots & \ddots & \vdots & \vdots & \vdots & \ddots & \vdots \\
							F_3 & F_5 & F_7 & \ldots & F_{n} & F_{n-2} & F_{n-4} & \ldots & F_1\\
								F_{1} & F_{3} & F_{5} & \dots &F_{n-2} & F_{n} & F_{n-2} & \dots & F_{1} \\
							F_{1} & F_{1} & F_{3} & \dots & F_{n-4} & F_{n-2} & F_{n} & \dots & F_{3} \\
							\vdots & \vdots & \vdots & \ddots & \vdots & \vdots & \vdots & \ddots & \vdots \\
							F_{n-2} & F_{n-4} & F_{n-6} & \dots &F_1 &  F_{1} & F_{3} & \dots & F_{n}
							\end{pmatrix}$$

\end{theorem}
	
	For the proof, we will need a preliminary result. Let $P_k$ be a path with $k$ vertices, and let $C(P_k)$ be its cone. Note that $C(P_k)$ can be equivalently characterized as the wheel graph $W_{k+1}$ on $k+1$ vertices with one boundary edge deleted.

\begin{lemma}\label{lem:spanforcone}
For any $k \geq 1$, the number of spanning trees of $C(P_k)$ is $F_{2k}$.
\end{lemma}

\begin{proof}
Let $\tau_k$ be the number of spanning trees of $C(P_k)$. It is easy to check that
$$\tau_1  = 1 = F_2, \indent \tau_2 = 3 = F_4$$
and that $f_k = F_{2k}$ satisfies
\begin{equation}\label{eq:identifier}
f_{k+1} = 3 f_k - f_{k - 1}, \indent \forall k \geq 2.
\end{equation}
By the uniqueness of solutions of difference equations, it therefore suffices to verify that $f_k = \tau_k$ satisfies (\ref{eq:identifier}) as well. 

For this, let $e_1, \ldots, e_{k}$ be the consecutive edges of $P_{k+1}$, viewed as edges in $C(P_{k+1})$. Viewing $C(P_{k+1})$ as $W_{k+2}$ with one edge deleted, let $e_s$ be the spoke that shares a vertex with $e_{k}$; see Figure \ref{figure3}.

\vspace{1cm}
\begin{figure}[h]
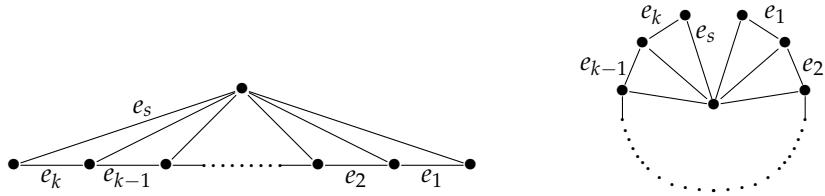

\centerline{\xy /r1.2pc/:,
(-3,0)*={\xy
		(-10,0)*{};
		(-6,0)*{\bullet}="A";
		(-4,0)*{\bullet}="B";
		(-2,0)*{\bullet}="B2";
		(-1,0)*{}="C1";
		(1,0)*{}="C2";
		(2,0)*{\bullet}="D2";
		(4,0)*{\bullet}="D";
		(6,0)*{\bullet}="E";
		(0,2)*{\bullet}="X";
		"A";"X" **\dir{-};
		"B";"X" **\dir{-};
		"B2";"X" **\dir{-};
		"D2";"X" **\dir{-};
		"D";"X" **\dir{-};
		"E";"X" **\dir{-};
		"A";"B" **\dir{-};
		"B";"B2" **\dir{-};
		"B2";"C1" **\dir{-};
		"C1";"C2"**\crv{~*=<4pt>{.} (0,0)};
		"C2";"D2" **\dir{-};
		"D2";"D" **\dir{-};
		"D";"E" **\dir{-};
(-5,-.35)*{e_k};
(-3,-.35)*{e_{k-1}};
(3,-.35)*{e_{2}};
(5,-.35)*{e_1};
(-2.6,1.5)*{e_s};
\endxy};
(11,1)*={\xy /r3pc/:,
		(0,0)*{\bullet}="v*";
		(.3,.94)*{\bullet}="v2";
		(.75,.65)*{\bullet}="v3";
		(.96,.15)*{\bullet}="v4";
		(.96,-.15)*{}="v5";
		(-.3,.94)*{\bullet}="w2";
		(-.75,.65)*{\bullet}="w3";
		(-.96,.15)*{\bullet}="w4";
		(-.96,-.15)*{}="w5";
		(-.65,.95)*{e_k};
		(-1.15,.4)*{e_{k-1}};
		(1.05,.4)*{e_{2}};
		(.65,.95)*{e_1};
		(-.08,.75)*{e_s};
		"v*";"v2" **\dir{-};
		"v*";"v3" **\dir{-};
		"v*";"v4" **\dir{-};
		"v2";"v3" **\dir{-};
		"v3";"v4" **\dir{-};
		"v4";"v5" **\dir{-};
		"v*";"w2" **\dir{-};
		"v*";"w3" **\dir{-};
		"v*";"w4" **\dir{-};
		"w2";"w3" **\dir{-};
		"w3";"w4" **\dir{-};
		"w4";"w5" **\dir{-};
		"w5";"v5" **\crv{~*=<4pt>{.} (-.8,-.8)&(0,-1)&(.8,-.8)};
\endxy};
\endxy}
\vspace{1cm}
\caption{Pictured here are two illustrations of $C(P_{k+1})$, with edges labeled as described above. On the left, the graph $C(P_{k+1})$ is viewed as the cone on a path with $k+1$ vertices (hence $k$ edges). The figure on the right illustrates $C(P_{k+1})$ as the wheel graph $W(P_{k+2})$ on $k+2$ vertices, with one edge removed. \hfill $\Diamond$}
\label{figure3}
\end{figure}

Now we can view $C(P_k)$ as the subgraph of $C(P_{k+1})$ obtained by deleting the edges $e_k$ and $e_s$, as well as their common vertex. From this perspective, the following hold by direct observation:

\begin{itemize}
\item There are $\tau_{k}$ spanning trees of $C(P_{k+1})$ that do not contain $e_s$. 

\item Of those that do contain $e_s$, there are $\tau_{k}$ that do not contain $e_k$. 

\item Of those that contain $e_s$ and $e_k$, there are $\tau_{k-1}$ that do not contain $e_{k-1}$. 

\item Of those that contain $e_s, e_k,$ and $e_{k-1}$, there are $\tau_{k-2}$ that do not contain $e_{k-2}$. 

 $\vdots$

\item Of those that contain $e_s, e_k, e_{k-1} , \ldots,$ and $e_{2}$, there are $\tau_1$ that do not contain $e_1$.

\item There is one spanning tree that contains $e_s, e_k, \ldots, e_2,$ and $e_1$. 
\end{itemize}

It follows that $\tau_{k+1} = \tau_k + \tau_k + ( \sum_{j = 1}^{k-1} \tau_j ) + 1$ for all $ k \geq 1$. Replacing $k$ with $k -1$, we obtain $\tau_{k} = \tau_{k -1 } + (\sum_{j = 1}^{k-1} \tau_j ) + 1$ for all $k \geq 2$. 
Combining these gives $\tau_{k+1} = \tau_k + \tau_k + \tau_k - \tau_{k - 1}$ for $k \geq 2$, which is (\ref{eq:identifier}). 
\end{proof}

\begin{proof}[Proof of Theorem \ref{thm:laplacewheel}] 
 As shown in \cite{BP}, the number of spanning trees of $W_{n+1}$ is $\det(L') = A_{2n}-2$. It then follows from Corollary \ref{cor:L'inverse} that the $(i, j)$-component of the inverse of $L'$ is given by 
$$(L')^{-1}_{ij} = \frac{1}{A_{2n} - 2}\vert \Ss_2(v_i, v_j) \vert. $$
Due to the cyclic symmetry of wheel graphs, it suffices to compute this under the assumption that $i = 1$ and $1 \leq j \leq n$. (In the matrix for $(L')^{-1}$ claimed in the statement of Theorem \ref{thm:laplacewheel}, this symmetry manifests itself as each row being a cyclic permutation of the previous.) Similarly, we may assume that the indexing is such that, as we proceed from $v_1$ to $v_j$, we have gone no more than half-way around the boundary cycle; that is, we may assume
\begin{equation}\label{eq:jineq}
j \leq n/2 +1.
\end{equation}
This can be achieved by either relabeling the vertices in the opposite direction, or reflecting over the line through $v_*$ and $v_1$, which amounts to the graph automorphism $v_i \mapsto v_{n+2-i}$. See Figure \ref{figure4}. (In the claimed matrix for $(L')^{-1}$, this reflection-symmetry manifests itself as the symmetry in the top row about the term $A_0$ when $n$ is even, and between the two $F_1$ terms when $n$ is odd.)

\begin{figure}[h]
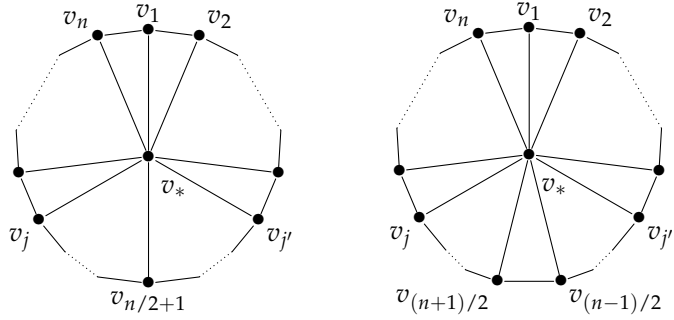

\centerline{\xy /r4pc/:,
(-1.5,0)*={\xy
		(0,0)*{\bullet}="v*";
		(0,1)*{\bullet}="v1";
		"v1"*+!D{\hbox{${{v_1}}$}};
		(.4,.95)*{\bullet}="v2";
		"v2"*+!DL{\hbox{${{v_2}}$}};
		(.7,.8)*{}="v3";
		(-.4,.95)*{\bullet}="vn";
		"vn"*+!DR{\hbox{${{v_n}}$}};
		(-.7,.8)*{}="vn1";
		(1.0227, -.1286)*{\bullet}="vjj1+";
		(1.0428, .2062)*{}="vjj15+";
		(.66, -.75)*{}="vjj1-";
		(.866, -.5)*{\bullet}="vjj";
		"vjj"*+!UL{\hbox{${{v_{j'}}}$}};
		(-1.0227, -.1286)*{\bullet}="vj1+";
		(-1.0428, .2062)*{}="vj15+";
		(-.66, -.75)*{}="vj1-";
		(-.866, -.5)*{\bullet}="vj";
		"vj"*+!UR{\hbox{${{v_{j}}}$}};
		(0, -1)*{\bullet}="vn2";
		"vn2"*+!U{\hbox{${{v_{n/2+1}}}$}};
		(.4,-.95)*{}="vn2-";
		(-.4,-.95)*{}="vn2+";
		"v*";"vn2" **\dir{-};
		"v*";"v1" **\dir{-};
		"v*";"v2" **\dir{-};
		"v*";"vn" **\dir{-};
		"v*";"vj" **\dir{-};
		"v*";"vj1+" **\dir{-};
		"v*";"vjj" **\dir{-};
		"v*";"vjj1+" **\dir{-};
		"v1";"v2" **\dir{-};
		"v1";"vn" **\dir{-};
		"v2";"v3" **\dir{-};
		"vn";"vn1" **\dir{-};
		"vjj15+";"vjj1+" **\dir{-};
		"vjj1+";"vjj" **\dir{-};
		"vjj1-";"vjj" **\dir{-};
		"vj15+";"vj1+" **\dir{-};
		"vj1+";"vj" **\dir{-};
		"vj1-";"vj" **\dir{-};
		"vn2-";"vn2" **\dir{-};
		"vn2+";"vn2" **\dir{-};
		"vjj1-";"vn2-" **\dir{.};
		"vj1-";"vn2+" **\dir{.};
		"v3";"vjj15+" **\dir{.};
		"vn1";"vj15+" **\dir{.};
		(.2,-.25)*{v_*};
\endxy};
(1.5,0)*={\xy
		(0,0)*{\bullet}="v*";
		(0,1)*{\bullet}="v1";
		"v1"*+!D{\hbox{${{v_1}}$}};
		(.4,.95)*{\bullet}="v2";
		"v2"*+!DL{\hbox{${{v_2}}$}};
		(.7,.8)*{}="v3";
		(-.4,.95)*{\bullet}="vn";
		"vn"*+!DR{\hbox{${{v_n}}$}};
		(-.7,.8)*{}="vn1";
		(1.0227, -.1286)*{\bullet}="vjj1+";
		(1.0428, .2062)*{}="vjj15+";
		(.66, -.75)*{}="vjj1-";
		(.866, -.5)*{\bullet}="vjj";
		"vjj"*+!UL{\hbox{${{v_{j'}}}$}};
		(-1.0227, -.1286)*{\bullet}="vj1+";
		(-1.0428, .2062)*{}="vj15+";
		(-.66, -.75)*{}="vj1-";
		(-.866, -.5)*{\bullet}="vj";
		"vj"*+!UR{\hbox{${{v_{j}}}$}};
		(.25, -1)*{\bullet}="vn-";
		"vn-"*+!UL{\hbox{${{v_{(n-1)/2}}}$}};
		(-.25, -1)*{\bullet}="vn+";
		"vn+"*+!UR{\hbox{${{v_{(n+1)/2}}}$}};
		(.5,-.91)*{}="vn2-";
		(-.5,-.91)*{}="vn2+";
		"v*";"vn+" **\dir{-};
		"v*";"vn-" **\dir{-};
		"v*";"v1" **\dir{-};
		"v*";"v2" **\dir{-};
		"v*";"vn" **\dir{-};
		"v*";"vj" **\dir{-};
		"v*";"vj1+" **\dir{-};
		"v*";"vjj" **\dir{-};
		"v*";"vjj1+" **\dir{-};
		"v1";"v2" **\dir{-};
		"v1";"vn" **\dir{-};
		"v2";"v3" **\dir{-};
		"vn";"vn1" **\dir{-};
		"vjj15+";"vjj1+" **\dir{-};
		"vjj1+";"vjj" **\dir{-};
		"vjj1-";"vjj" **\dir{-};
		"vj15+";"vj1+" **\dir{-};
		"vj1+";"vj" **\dir{-};
		"vj1-";"vj" **\dir{-};
		"vn2-";"vn-" **\dir{-};
		"vn2+";"vn+" **\dir{-};
		"vn-";"vn+" **\dir{-};
		"vjj1-";"vn2-" **\dir{.};
		"vj1-";"vn2+" **\dir{.};
		"v3";"vjj15+" **\dir{.};
		"vn1";"vj15+" **\dir{.};
		(.2,-.25)*{v_*};
\endxy};
\endxy}
\vspace{2cm}
\caption{Illustrated here is the case where $j > n/2 + 1$. The figure on the left (resp. right) is for $n$ even (resp. odd). Notice that in both pictures, reflecting across the vertical line from $v_*$ to $v_1$ sends $v_j$ to $v_{j'}$, where $j' \defeq n+2-j< n/2+1$. \hfill $\Diamond$}
\label{figure4}
\end{figure}

Fix $(\T_0, \T_1 ) \in \Ss_2(v_1, v_j)$. Since $v_* \in \T_0$, it follows that $\T_1$ will always be a path along the boundary cycle, containing $v_1, v_j$. In particular, the tree $\T_1$ contains at least $j$ vertices (by (\ref{eq:jineq})), and at most $n$ vertices. For $j \leq m \leq n$, write $\Ss_2^m$ for the set of $(\T_0, \T_1 )\in  \Ss_2(v_1, v_j)$ with the property that $\T_1$ contains exactly $m$ vertices. Then it suffices to determine the size of each $\Ss_2^m$, since
$$\vert \Ss_2(v_1, v_j) \vert= \sumdd{m = j}{n} \vert \Ss_2^m \vert.$$

\begin{figure}[h]
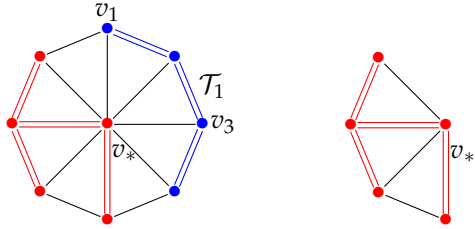

\vspace{2cm}
\centerline{\xy /r3pc/:,
(-1.5,0)*={\xy
	(0,0)*[red]{\bullet}="A0";
		(0,1)*[blue]{\bullet}="A1";
		(.7071,.7071)*[blue]{\bullet}="A2";
		(1,0)*[blue]{\bullet}="A3";
		(.7071,-.7071)*[blue]{\bullet}="A4";
		(0,-1)*[red]{\bullet}="A5";
		(-.7071,-.7071)*[red]{\bullet}="A6";
		(-1,0)*[red]{\bullet}="A7";
		(-.7071,.7071)*[red]{\bullet}="A8";
		"A1"*+!D{\hbox{${{v_1}}$}};
		"A3"*+!L{\hbox{${{v_3}}$}};
		"A0";"A1" **\dir{-};
		"A0";"A2" **\dir{-};
		"A0";"A3" **\dir{-};
		"A0";"A4" **\dir{-};
		"A0";"A5" **[red]\dir2{-};
		"A0";"A6" **\dir{-};
		"A0";"A7" **[red]\dir2{-};
		"A0";"A8" **\dir{-};
		"A1";"A2" **[blue]\dir2{-};
		"A2";"A3" **[blue]\dir2{-};
		"A3";"A4" **[blue]\dir2{-};
		"A4";"A5" **\dir{-};
		"A5";"A6" **\dir{-};
		"A6";"A7" **[red]\dir2{-};
		"A7";"A8" **[red]\dir2{-};
		"A8";"A1" **\dir{-};
		(.18,-.3)*{v_*};
		(1.1,.4)*{\T_1};
\endxy};
(1.5,-.3)*={\xy
	(0,0)*[red]{\bullet}="A0";
		(0,-1)*[red]{\bullet}="A5";
		(-.7071,-.7071)*[red]{\bullet}="A6";
		(-1,0)*[red]{\bullet}="A7";
		(-.7071,.7071)*[red]{\bullet}="A8";
		"A0";"A5" **[red]\dir2{-};
		"A0";"A6" **\dir{-};
		"A0";"A7" **[red]\dir2{-};
		"A0";"A8" **\dir{-};
		"A5";"A6" **\dir{-};
		"A6";"A7" **[red]\dir2{-};
		"A7";"A8" **[red]\dir2{-};
		(.18,-.3)*{v_*};
\endxy};
\endxy}
\vspace{1.8cm}
\caption{The figure on the left is the wheel graph $W_9$ on nine vertices, with $j = 3$. A spanning 2-forest $\{\T_0, \T_1 \}$ is indicated, with $\T_0$ in red and $\T_1$ in blue; the edges of these trees have been doubled for emphasis. The tree $\T_1$ has $m = 4$ vertices given by $v_1, \ldots, v_4$. Note that $\T_1$ is a path consisting entirely of boundary vertices. On the right is the cone $C(P_{4})$ obtained by deleting all vertices and edges associated to $\T_1$. Note that the trees $\T_0$ for which $(\T_0, \T_1) \in \Ss_2^4$ are in one-to-one correspondence with the spanning trees of $C(P_4)$ on the right.\hfill $\Diamond$}
\label{figure5}
\end{figure}

Fix $m$ with $j \leq m \leq n-(j-1)$. Then any $( \T_0, \T_1 ) \in \Ss_2^m$ must be such that $\T_1$ contains the smaller of the two arcs in the boundary $k$-cycle that connects $v_1$ and $v_j$. See Figure \ref{figure5} for an example and Figure \ref{figure6} for a non-example. There are $m-(j-1)$ choices of such $\T_1$, and we fix one $\T_1$. Remove from $W_{n+1}$ all vertices of $\T_1$ as well as all edges adjacent to a vertex in $\T_1$. The result is the cone $C(P_{n-m})$ of a path $P_{n-m}$ on $n-m$ vertices. Then any $\T_0$ with $(\T_0, \T_1 ) \in \Ss_2^m$ is a spanning tree for $C(P_{n-m})$. By Lemma \ref{lem:spanforcone}, there are $F_{2(n-m)}$ such spanning trees. This gives $\vert \Ss_2^m \vert = (m-(j-1)) F_{2(n-m)}$.

When $n-(j-2) \leq m \leq n-1 $, a boundary path $\T_1$ with $m$ vertices can contain either of the two boundary arcs between $v_1$ and $v_j$. See Figure \ref{figure6}. There are therefore $2m - n$ such paths $\T_{1}$. Fixing such a path $\T_1$, just as in the previous paragraph, there are $F_{2(n-m)}$ choices of $\T_0$ so that $(\T_0, \T_1) \in \Ss_2^m$. This gives $\vert \Ss_2^m \vert =(2m-n)F_{2(n-m)}$.

\vspace{1.2cm}
\begin{figure}[h]
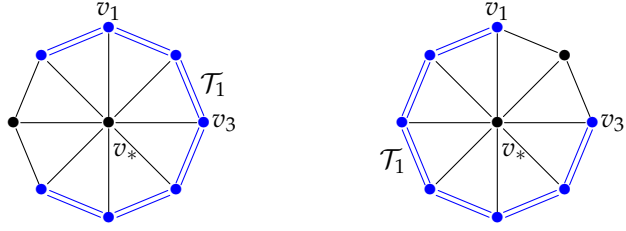

\centerline{\xy /r3pc/:,
(-2,0)*={\xy
	(0,0)*{\bullet}="A0";
		(0,1)*[blue]{\bullet}="A1";
		(.7071,.7071)*[blue]{\bullet}="A2";
		(1,0)*[blue]{\bullet}="A3";
		(.7071,-.7071)*[blue]{\bullet}="A4";
		(0,-1)*[blue]{\bullet}="A5";
		(-.7071,-.7071)*[blue]{\bullet}="A6";
		(-1,0)*{\bullet}="A7";
		(-.7071,.7071)*[blue]{\bullet}="A8";
		"A1"*+!D{\hbox{${{v_1}}$}};
		"A3"*+!L{\hbox{${{v_3}}$}};
		"A0";"A1" **\dir{-};
		"A0";"A2" **\dir{-};
		"A0";"A3" **\dir{-};
		"A0";"A4" **\dir{-};
		"A0";"A5" **\dir{-};
		"A0";"A6" **\dir{-};
		"A0";"A7" **\dir{-};
		"A0";"A8" **\dir{-};
		"A1";"A2" **[blue]\dir2{-};
		"A2";"A3" **[blue]\dir2{-};
		"A3";"A4" **[blue]\dir2{-};
		"A4";"A5" **[blue]\dir2{-};
		"A5";"A6" **[blue]\dir2{-};
		"A6";"A7" **\dir{-};
		"A7";"A8" **\dir{-};
		"A8";"A1" **[blue]\dir2{-};
		(.18,-.3)*{v_*};
		(1.1,.4)*{\T_1};
\endxy};
(2,0)*={\xy
	(0,0)*{\bullet}="A0";
		(0,1)*[blue]{\bullet}="A1";
		(.7071,.7071)*{\bullet}="A2";
		(1,0)*[blue]{\bullet}="A3";
		(.7071,-.7071)*[blue]{\bullet}="A4";
		(0,-1)*[blue]{\bullet}="A5";
		(-.7071,-.7071)*[blue]{\bullet}="A6";
		(-1,0)*[blue]{\bullet}="A7";
		(-.7071,.7071)*[blue]{\bullet}="A8";
		"A1"*+!D{\hbox{${{v_1}}$}};
		"A3"*+!L{\hbox{${{v_3}}$}};
		"A0";"A1" **\dir{-};
		"A0";"A2" **\dir{-};
		"A0";"A3" **\dir{-};
		"A0";"A4" **\dir{-};
		"A0";"A5" **\dir{-};
		"A0";"A6" **\dir{-};
		"A0";"A7" **\dir{-};
		"A0";"A8" **\dir{-};
		"A1";"A2" **\dir{-};
		"A2";"A3" **\dir{-};
		"A3";"A4" **[blue]\dir2{-};
		"A4";"A5" **[blue]\dir2{-};
		"A5";"A6" **[blue]\dir2{-};
		"A6";"A7" **[blue]\dir2{-};
		"A7";"A8" **[blue]\dir2{-};
		"A8";"A1" **[blue]\dir2{-};
		(.18,-.3)*{v_*};
		(-1.1,-.4)*{\T_1};
\endxy};
\endxy}
\vspace{1.8cm}
\caption{Each of the figures above illustrate an example of $\T_1$ on $W_9$ with $j = 3$ and $m = 7$; the vertices and edges of $\T_1$ are indicated in blue, with the edges doubled for emphasis. The figure on the left contains the small arc in the boundary circle between $v_1$ and $v_j$. The figure on the right contains the large arc between $v_1$ and $v_j$.\hfill $\Diamond$}
\label{figure6}
\end{figure}

The last case to consider is when $m = n$. If $(\T_0, \T_1) \in \Ss_2^n$, then $\T_0$ necessarily consists only of $v_*$, and $\T_1$ can be any spanning tree of the boundary $n$-cycle obtained by deleting $v_*$ from $W_{n+1}$. Thus, $\vert \Ss_2^n \vert = n$. In summary, we have
\[ \vert \Ss_2(v_1, v_j)\vert  = n + \sum_{m=j}^{n-(j-1)} (m-(j-1))F_{2(n-m)}
+ \sum_{m=n-(j-2)}^{n-1} (2m - n)F_{2(n-m)}  \]

The aim now is to simplify this. At this point, it is convenient to change the summation index from $m$ to $\ell \defeq n - m$. Then using the identities
\[ \begin{array}{rcl}
\sumdd{\ell=a}{b}\ell F_{2\ell} &=& bF_{2b+1}-F_{2b}-aF_{2(a-1)+1} + F_{2(a-1)+2} \\ \sumdd{\ell=a}{b}F_{2\ell}& = &F_{2b+1}-F_{2(a-1)+1} 
\end{array}\]
we can write
$$\begin{array}{rcl}
\vert \Ss_2(v_1, v_j)\vert  & = & n + \sumdd{\ell=j-1}{n-j} (n-\ell-j+1)F_{2\ell} + \sumdd{\ell=1}{j-2} (n-2\ell)F_{2\ell} \\
& = & n + (n - j + 1)\sumdd{\ell=j-1}{n-j} F_{2\ell}  - \sumdd{\ell=j-1}{n-j} \ell F_{2\ell}+  n \sumdd{\ell=1}{j-2} F_{2\ell} - 2  \sumdd{\ell=1}{j-2} \ell F_{2\ell} \\
&=& n+(n-j+1)[F_{2(n-j)+1} -F_{2(j-2)+1}] \\ 
&& - [(n-j)F_{2(n-j)+1} -F_{2(n-j)} -(j-1)F_{2(j-2)+1} +F_{2(j-2)+2}] \\ &&+n[F_{2(j-2)+1} - F_{2(0)+1}] \\ &&-2[(j-2)F_{2(j-2)+1}-F_{2(j-2)} - (1)F_{2(0)+1} +F_{2(0)+2}] \\
    &=& F_{2(n-j)+1} + F_{2(n-j)} - F_{2(j-2)+2} + 2 F_{2(j-2)+1}+2 F_{2(j-2)}\\
    & = & F_{2(n-j+1)} +F_{2(j-1)}
\end{array}$$
where, in the last line, we used the recursive relation for the Fibonacci numbers. The theorem is an immediate consequence of the following identity:

\medskip

\noindent \emph{Claim:} $\fracd{1}{A_{2n}-2}(F_{2(n-j+1)} + F_{2(j-1)}) = \left\{ \begin{array}{ll}
                             \fracd{1}{5 F_n}A_{n-2(j-1)}& \textrm{if $n$ is even}\\
                             \fracd{1}{A_n}F_{n-2(j-1)}& \textrm{if $n$ is odd}\\
                             \end{array}\right.$
                             
                             \medskip
  
  To prove this claim, recall the Fibonacci and Lucas numbers satisfy $F_{-a}  =  (-1)^{a+1} F_a$, $A_{-a} =  (-1)^a A_a$, and                           
\begin{equation}\label{eq:fibadd}
\begin{array}{rclcrcl}
F_{a+b} &=& \frac{1}{2}(F_a A_b + A_a F_b) &\hspace{1cm}&  A_{a+b} &=& \frac{1}{2}(5F_a F_b + A_a A_b).
\end{array}
\end{equation}
This implies
\begin{equation} \label{eq:fibsub}
F_{a-b} = \frac{(-1)^b}{2}(F_a A_b - A_a F_b) \hspace{1cm}  A_{a-b} = \frac{(-1)^{b+1}}{2}(5F_a F_b - A_a A_b) 
\end{equation}
From which we obtain the identities $A_a F_b = F_{a+b} + (-1)^{b+1}F_{a-b}$ and $A_a A_b = A_{a+b} + (-1)^b A_{a-b}$. Freely referring to these identities, the following computation completes the proof of the claim when $n$ is even:
$$\begin{array}{rcl}
(A_{2n}-2)A_{n-2(j-1)} & = & (A_{2n}+2) A_{n-2(j-1)} - 4 A_{n-2(j-1)}\\
& = & A_n^2 A_{n-2(j-1)} - 4 A_{n - 2(j-1)}\\
& = & A_n(A_{2(n-j+1)} + A_{2(j-1)}) - 4 A_{n - 2(j-1)}\\
& = & A_n A_{2(n-j+1)} + A_n A_{2(j-1)} - 4 A_{n - 2(j-1)}\\
& = & 2A_{n-2(j-1)} + 5 F_n F_{2(n-j+1)}\\
&& + 2A_{n-2(j-1)} + 5 F_n F_{2(j-1)} - 4A_{n-2(j-1)}\\
& = & 5 F_n(F_{2(n-j+1)} + F_{2(j-1)}).
\end{array}$$
The case where $n$ is odd is similar and left to the reader. 
\end{proof}

\subsubsection{Proofs of Theorem \ref{thm:wheel} and Corollary \ref{cor:specialwheel}} \label{sec:WheelProofs}

\begin{proof}[Proof of Theorem \ref{thm:wheel}]
We prove the result under the assumption that $n$ is even; the case where $n$ is odd is similar. The wheel graph is the cone of a regular graph (a cycle graph), so by Corollary \ref{cor:2} it suffices to show that $\sigma$ is in the image of $L'$ if and only if (\ref{eq:neven}) holds for all $i$. Define $M \defeq 5F_n (L')^{-1}$; this is an integral matrix by Theorem \ref{thm:laplacewheel}. Then $\sigma$ is in the image of $L': \bb{Z}^{\V'} \rightarrow \bb{Z}^{\V'}$ if and only if $5F_n$ divides the $i$th component $\langle M \sigma, e_i \rangle$ of $M\sigma$ for all $i \in \bb{Z}/n \bb{Z}$. Using Theorem \ref{thm:laplacewheel} again, we see that $\langle M \sigma, e_{i+n/2} \rangle$ is the left-hand side of (\ref{eq:neven}). This finishes the proof since $5F_n$ divides $\langle M \sigma, e_i \rangle$ for all $i \in \bb{Z}/n \bb{Z}$ if and only if $5F_n$ divides $\langle M \sigma, e_{i+n/2} \rangle$ for all $i \in \bb{Z}/n \bb{Z}$.
\end{proof}

Our proof of Corollary \ref{cor:specialwheel} relies on the following equivalences for even Lucas numbers.

\begin{lemma} \label{lem:mod A conj}
Let $j,\ell, c \in \bb{Z}$ with $\ell\geq 1$. Then
$$\begin{array}{rcll} 
    A_{2j}  & \equiv & (-1)^{j}A_{0} &\mod 5\\
    A_{2(j + c 2^{\ell})} & \equiv & (-1)^c A_{2j}& \mod A_{2^{\ell}}.
\end{array}$$
\end{lemma}
\begin{proof}
Add (\ref{eq:fibadd}) and (\ref{eq:fibsub}) with $a = b = j$ to get $A_{2j} =  5 F_j^2+(-1)^j A_0$. The first identity of the lemma follows. For the second identity, use $A_{a+b} = A_a A_b - (-1)^b A_{a-b}$ with $a = 2j+2^\ell$ and $b =  2^\ell$ to get
$$A_{2(j + 2^{\ell})} =A_{2j +  2^{\ell} + 2^{\ell}}  =  A_{2j+ 2^{\ell}}A_{2^{\ell}} - A_{2j} \equiv - A_{2j} \mod A_{2^\ell}.$$
This proves the identity for $c = 1$, and the identity for general $c \in \bb{Z}$ follows from induction. 
\end{proof}

\begin{proof}[Proof of Corollary \ref{cor:specialwheel}]
Assume $n = 2^k$ for $k \geq 2$, fix $\sigma \in \bb{Z}^{\V'}$ and set 
$$\begin{array}{rcl}
D_{i} &\defeq &   A_{n}\sigma_{i+2^{k-1}} + A_{0}\sigma_{i} + \sumdd{m=1}{2^{k-1}-1} A_{2m}(\sigma_{i+m}+\sigma_{i-m}).
\end{array}$$ 
As discussed in the introduction to this section, it follows from Theorem \ref{thm:laplacewheel} and the relatively prime decomposition $5F_{2^k} = 5 \prod_{\ell = 1}^{k-1} A_{2^\ell}$ that $\sigma$ is immutable if and only if, for all $i \in \bb{Z}/2^k \bb{Z}$, the integer $D_{ i}$ is congruent to 0 mod 5, and congruent to 0 mod $A_{2^\ell}$ for all $1 \leq \ell \leq k - 1$. Our aim is to simplify these conditions by exploiting the symmetries expressed in Lemma \ref{lem:mod A conj}.

We begin with a preliminary computation. By the $2^k$-periodicity of the index of $\sigma_i$, we have $\sigma_{i-m} = \sigma_{i+ 2^k-m}$, and so we can write
$$\begin{array}{rcl}
D_{ i} & = & \Big( A_0 \sigma_i + \sumdd{m =1}{2^{k-1}-1} A_{2m} \sigma_{i+m}\Big) + \Big(A_{2^k} \sigma_{i+2^k-2^{k-1}} + \sumdd{m = 1}{2^{k-1}-1}A_{2m} \sigma_{i+2^k-m}\Big)\\
& = & \sumdd{m = 0}{2^{k-1}-1} A_{2m} \sigma_{i+m} + \sumdd{m = 1}{2^{k-1}}A_{2m} \sigma_{i+2^k -m}\\
& = & \sumdd{m = 0}{2^{k-1}-1} A_{2m} \sigma_{i+m} + \sumdd{m = 2^{k-1}}{2^{k}-1}A_{2(2^k-m)} \sigma_{i+m}
\end{array}$$
where we performed the change of index $m \mapsto  2^{k}-m$ in the second sum of the last line. It follows from Lemma \ref{lem:mod A conj} that $A_{2(2^k-m)} \equiv A_{-2m} = A_{2m}$, whenever working mod $5$ or mod $A_{2^\ell}$. Thus, we can continue the above to get
\begin{equation}\label{eq:Dkmod}
D_{i} \equiv \sumdd{m=0}{2^k-1} A_{2m} \sigma_{i+m} \indent \textrm{mod $5$ \; \; or \; \; mod $A_{2^\ell}$}.
\end{equation}
This is the desired preliminary computation.

Now consider the mod 5 case. Applying Lemma \ref{lem:mod A conj} again, we see immediately from (\ref{eq:Dkmod}) that
$$D_i \equiv A_0 \sumdd{m=0}{2^k-1} (-1)^m \sigma_{i+m} \indent \mod 5.$$
Note first that this shows $D_{i+1} \equiv -D_i \mod 5$, and so this being congruent to 0 is independent of $i$. Since $A_0 = 2$ is invertible mod $5$, it follows that $D_{i} \equiv 0 \mod 5$ for all $i$ if and only if (\ref{eq:mod5spw}) holds (which corresponds to $i = 0$). 

Next, fix $1 \leq \ell \leq k -1$ and work mod $A_{2^\ell}$. By (\ref{eq:Dkmod}), we have
$$\begin{array}{rcl}
D_{i+2^\ell} & \equiv & \sumdd{m=0}{2^{k}-1} A_{2m} \sigma_{i+m+2^\ell}\\
& \equiv & \sumdd{m=0}{2^{k}-2^\ell-1} A_{2m} \sigma_{i+m+2^\ell} + \sumdd{m=2^k-2^\ell}{2^{k}-1} A_{2m} \sigma_{i+m+2^\ell}\\
& = & \sumdd{m=2^\ell}{2^{k}-1} A_{2(m-2^\ell)} \sigma_{i+m} + \sumdd{m=0}{2^{\ell}-1} A_{2(m-2^\ell-2^k)} \sigma_{i+m}
\end{array}$$
where we did a change of index in each sum appearing in the last line. By Lemma \ref{lem:mod A conj}, we can continue this as
$$D_{i+2^\ell} \equiv -\sumdd{m=2^\ell}{2^{k}-1} A_{2m} \sigma_{i+m} - \sumdd{m=0}{2^{\ell}-1} A_{2m} \sigma_{i+m} = -D_{i} \indent \mod A_{2^\ell}.$$
Thus $D_{i} \equiv 0 \mod A_{2^\ell}$ for all $i \in \{0, \ldots, 2^k-1\}$ if and only if $D_{i} \equiv 0 \mod A_{2^\ell}$ for all $i \in \{0, \ldots, 2^\ell-1\}$ (the latter being a considerably smaller system). 

Finally, we need to compute $D_{i}$ mod $A_{2^\ell}$ and show it has the claimed form. For this, return to the expression (\ref{eq:Dkmod}) and note that each $m \in \{0, \ldots, 2^k -1\}$ can be expressed as $m = 2^\ell c + j$ for unique $c \in \{0, \ldots, 2^{k-\ell }-1 \}$ and $j \in \{0, \ldots, 2^\ell - 1 \}$. Then using Lemma \ref{lem:mod A conj} again, we have
$$\begin{array}{rcl}
D_i & \equiv  & \sumdd{c = 0}{2^{k-\ell}-1} \sumdd{j = 0}{2^\ell - 1} A_{2(2^\ell c + j)} \sigma_{i + 2^\ell c + j}\\
& \equiv  & \sumdd{c = 0}{2^{k-\ell}-1} (-1)^c \sumdd{j = 0}{2^\ell - 1} A_{2j} \sigma_{i + 2^\ell c + j}
\end{array}$$
Note that when $j = 2^{\ell - 1}$ we have $A_{2j} \equiv 0$. Thus, we can write
$$\begin{array}{rcl}
D_i & \equiv  & \sumdd{c = 0}{2^{k-\ell}-1} (-1)^c\Big[ A_0 \sigma_{i + 2^\ell c}\\
&& + \sumdd{j = 1}{2^{\ell-1} - 1} A_{2j} \sigma_{i + 2^\ell c + j}+\sumdd{j = 2^\ell +1}{2^{\ell} - 1} A_{2j} \sigma_{i + 2^\ell c + j}\Big]\\
& \equiv & \sumdd{c = 0}{2^{k-\ell}-1} (-1)^c\Big[ A_0 \sigma_{i + 2^\ell c}\\
&& + \sumdd{j = 1}{2^{\ell-1} - 1} A_{2j} \sigma_{i + 2^\ell c + j}+\sumdd{j = 1}{2^{\ell-1} - 1} A_{2(2^{\ell-1}+j)} \sigma_{i + 2^\ell c +2^{\ell-1}-  j}\Big]\\
\end{array}$$
where we did an index change in the second sum. Use $A_{2(2^{\ell-1} +j)} \equiv - A_{2(2^{\ell-1} -j)}$ and another reindexing to continue this as
$$\begin{array}{rcl}
D_i & \equiv & \sumdd{c = 0}{2^{k-\ell}-1} (-1)^c\Big[ A_0 \sigma_{i + 2^\ell c}\\
&& + \sumdd{j = 1}{2^{\ell-1} - 1} A_{2j} \sigma_{i + 2^\ell c + j}+\sumdd{j = 1}{2^{\ell-1} - 1} (-1)A_{2j} \sigma_{i + 2^\ell (c+1) -  j}\Big]\\
\end{array}$$
Focus on the last term (its sum over both $c$ and $j$), and write this as
$$\begin{array}{rcl}
\sumdd{c = 0}{2^{k-\ell}-1} (-1)^c \sumdd{j = 1}{2^{\ell-1} - 1} (-1)A_{2j} \sigma_{i + 2^\ell (c+1) -  j} & = & \sumdd{c = 0}{2^{k-\ell}-1} (-1)^{c+1} \sumdd{j = 1}{2^{\ell-1} - 1} A_{2j} \sigma_{i + 2^\ell (c+1) -  j}\\
& = & \sumdd{c = 1}{2^{k-\ell}} (-1)^c \sumdd{j = 1}{2^{\ell-1} - 1} A_{2j} \sigma_{i + 2^\ell c -  j}\\
& \equiv & \sumdd{c = 0}{2^{k-\ell}-1} (-1)^c \sumdd{j = 1}{2^{\ell-1} - 1} A_{2j} \sigma_{i + 2^\ell c -  j}
\end{array}$$
where the last equality holds by the $2^k$-periodicity of the index of $\sigma_i$, which implies the $c = 0$ term equals the $c = 2^\ell$ term. In summary, this gives
$$\begin{array}{rcl}
D_i & \equiv & \sumdd{c = 0}{2^{k-\ell}-1} (-1)^c\Big[ A_0 \sigma_{i + 2^\ell c} + \sumdd{j = 1}{2^{\ell-1} - 1} A_{2j} (\sigma_{i + 2^\ell c + j}+\sigma_{i+2^\ell c - j} )\Big].
\end{array}$$

\end{proof}

\end{document}